\documentclass[12pt]{amsart}
\usepackage{geometry}
\usepackage{graphicx}				
\usepackage{indentfirst}
\usepackage{hyperref}
\usepackage{ulem}	
\usepackage{amsthm}
\usepackage{amsfonts,bm}
\usepackage{amsmath}
\allowdisplaybreaks[4]
\usepackage{amscd}
\usepackage{enumerate,footnote}
\usepackage{amssymb}	
\usepackage{verbatim}
\usepackage{cancel}
\usepackage{cite}
\usepackage{textcomp}								
\usepackage{color}
\usepackage{mathrsfs}
\usepackage{appendix}

\newtheorem{example}{Example}[section]							
\newtheorem{theorem}{Theorem}[section]

\newtheorem{proposition}{Proposition}[section]
\newtheorem{lemma}{Lemma}[section]
\newtheorem{remark}{Remark}[section]

\newtheorem{corollary}{Corollary}[section]

\newtheorem{conjecture}{Conjecture}

\def\rd{\mathrm d}

\def\be{\begin{equation}}
\def\ee{\end{equation}}

\title[On $r$-neutralized entropy]{On $r$-neutralized entropy: entropy formula and existence of measures attaining the supremum}
\author{Changguang Dong}
\address{Chern Institute of Mathematics and LPMC, Nankai University, Tianjin 300071, China}
\email{dongchg@nankai.edu.cn}

\author{Qiujie Qiao}
\address{Chern Institute of Mathematics and LPMC, Nankai University, Tianjin 300071, China}
\email{qiujieqiao@mail.nankai.edu.cn}
\date{\today}


\begin{document}
\maketitle
\begin{abstract}
In this article we study $r$-neutralized local entropy 
and derive some entropy formulas. 
For an ergodic hyperbolic measure of a smooth system, we  show that
the $r$-neutralized local entropy equals the Brin-Katok local entropy plus
 $r$ times the pointwise dimension of the measure. 
We further establish the existence of ergodic measures  that maximize the $r$-neutralized entropy for certain hyperbolic systems. 
Moreover, we construct a uniformly hyperbolic system, for which such measures are not  unique. Finally, we present some rigidity results related to these ergodic measures. 
\end{abstract}

\tableofcontents

\section{Introduction and results}

A central topic in studying a given dynamical system is characterizing its complexity. 
From topological, ergodic and dimensional perspectives, 
various important invariants are introduced and extensively studied. 
These include 
topological entropy \cite{AKM, Bo73, [FH12],HHW17}, 
metric entropy \cite{Kolmogorov,[BK83], HHW17}, 
Lyapunov exponents \cite{Oseledec,LY-85-annals-1} and 
pointwise dimension \cite{LY-85-annals-2,BPS-99-annals, BW03,BW06} etc. 

Pointwise dimension reflects local fractal properties of invariant measures in dynamical systems. 
For a given measure $\mu$ and a point $x$, the pointwise dimension is defined as
\[ 
d_\mu(x) = \lim_{r \to 0} \frac{\log \mu(B(x,r))}{\log r}, 
\]
where $B(x, r)$ denotes a ball centered at $x$ with radius $r$. 
However, in many cases, the limit defining the pointwise dimension as $r\to 0$ does not exist. 
To address this issue,  
 unstable pointwise dimension $d^u_\mu(x)$ and 
stable pointwise dimension $d^s_\mu(x)$ are introduced and studied \cite{LY-85-annals-2,BPS-99-annals,ORH23} 
(see also Section \ref{differentiable dynamics} for definitions).  
Barreira, Pesin, and Schmeling established the existence of the pointwise dimension $d_\mu(x)$
in the case when $\mu$ is hyperbolic. 
These notions allow for a more detailed analysis of the chaotic behavior 
and complexity of dynamical systems.

In this work, we are interested in both the metric entropy and pointwise dimensions. 
As we know, the metric entropy is a global invariant which measures the rate at which 
information is produced in the system over time. 
Conversely, the pointwise dimension is a local quantity which 
measures the local fractal dimension at that point. 
Another huge difference between these two is that
 the Hausdorff dimension of the invariant measure equals to the 
essential supremum of the pointwise dimension of the measures 
in an ergodic decomposition\cite{BW06}, 
 rather than the integral of pointwise dimension
(as is the case with the metric entropy). 
Therefore, we wish to introduce a new type of entropy, 
which aims to establish a balance between the metric entropy and 
the integral of pointwise dimensions, analogous to the pressure.

Recently, some progress has been made in this direction. 
In \cite{ORH23}, 
to overcome phasing-out sub-exponential effects, 
Ben Ovadia and Rodriguez Hertz introduced the 
neutralized local entropy, defined as
\[
  \mathcal{E}_{\mu}(f,x)=\lim _{r \rightarrow 0} \varlimsup_{n \rightarrow \infty} -\frac{1}{n} \log \mu\left(B\left(x, n, e^{-nr}\right)\right).
\]
It coincides with the Brin-Katok local entropy almost everywhere for smooth systems. 
They provided a lower bound for the pointwise dimension of invariant measures in \cite{ORH23} 
and 
proved a dichotomy theorem for ``almost'' exponentially mixing volume in smooth systems in \cite{ORH23-2}. 
Subsequently, 
Ben Ovadia developed a bound on the gap, related to the volume growth, 
between any two consecutive conditional entropies\cite{O24}. 
We  observe that the neutralized local entropy may contain more information if not taking limit on $r$. 
This observation motivates us to study the dependence of the limit on $r$.

In this regard, 
we introduce the $r$-neutralized local entropy and  related definitions. 
We find that, the $r$-neutralized local entropy is indeed a linear combination of  
the metric entropy and pointwise dimensions. 




\subsection{$r$-neutralized local entropy formula}

Let us first introduce some definitions.

Let $(X,d)$ be a compact metric space, $f:X\to X$ a homeomorphism. 
 Let $h_{\text{top}}(f)$ be the topological entropy of $(X,d,f)$
and $\mathcal{M}(f, X)$ the set of $f$-invariant Borel probability measures on $X$. 
For each $\nu\in\mathcal{M}(f,X)$, we use $h_\nu(f)$ to denote its measure-theoretic entropy (namely metric entropy). 
The variational principle (c.f. \cite{[PW]}) states that
\begin{align*}
  h_{\text{top}}(f)=\sup\left\{h_{\nu}(f): \nu \in \mathcal{M}(f,X)\right\}.
\end{align*}
A measure $\nu\in \mathcal{M}(f,X)$ is called the {\bf measure of maximal entropy}, or {\bf MME} if $h_\mu(f)=h_{\text{top}}(f)$. 
For an $f$-invariant measure $\nu$, 
 we use $\operatorname{dim}_H\nu$ to denote its Hausdorff dimension (see subsection \ref{Pressure and some facts} for the definition).
Similarly, a measure $\mu\in \mathcal{M}(f,X)$ is called the {\bf measure of maximal Hausdorff dimension}, or {\bf MMHD}, if
\begin{align*}
  \operatorname{dim}_H\mu=\sup\{\operatorname{dim}_H\nu: \nu\in \mathcal{M}(f,X)\}.
\end{align*}
 Given $n\in \mathbb{N}$, $r>0$ and $x\in X$, 
the {\bf $r$-neutralized Bowen ball} $B(x,n,e^{-nr})$ and 
the Bowen ball $B(x,n,r)$ are defined respectively by 
  \begin{align*}
    B\left(x, n, e^{-nr}\right)&:=\left\{y \in X: d_n^f\left(x, y\right) < e^{-nr}, \forall\, 0 \leq i < n \right\},\\
    B\left(x, n, r\right)&:=\left\{y \in X: d\left(f^{i}(y), f^{i}(x)\right) < r, \forall\, 0 \leq i < n \right\}. 
  \end{align*}
Fix $\mu\in \mathcal{M}(f,X)$, according to the main theorem in \cite{[BK83]}, for $\mu$-a.e. $x\in X$, 
\begin{align*}
  \lim_{\epsilon\to 0}\varlimsup_{n\to \infty}-\frac{1}{n}\log \mu(B(x,n,\epsilon))=\lim_{\epsilon\to 0}\varliminf_{n\to \infty}-\frac{1}{n}\log \mu(B(x,n,\epsilon)).
\end{align*}
We use {\bf Brin-Katok local entropy} $h_\mu(f,x)$ to denote the above value. 
Additionally, 
the function $h_\mu(f,\cdot)$ is $f$-invariant and $\int_X h_\mu(f,x)\,\rd\mu=h_\mu(f)$. 

 For each $\mu\in \mathcal{M}(f,X)$, 
the {\bf $r$-neutralized local entropy} $h_{\mu,d}^{r}(f,x)$ and 
  the {\bf lower $r$-neutralized local entropy} $\underline{h}_{\mu,d}^{r}(f,x)$ 
  are defined respectively by
  \begin{align*}
    h_{\mu,d}^{r}(f,x):=\varlimsup_{n\to\infty}-\frac{1}{n}\log \mu (B(x,n,e^{-nr})), \quad \underline{h}_{\mu,d}^{r}(f,x):=\varliminf_{n\to\infty}-\frac{1}{n}\log \mu (B(x,n,e^{-nr})). 
  \end{align*}
Obviously, they are $f$-invariant, and are constant $\mu$-a.e. if $\mu$ is ergodic. 
Then we define the {\bf $r$-neutralized entropy} $h_{\mu,d}^{r}(f)$ and
the {\bf lower $r$-neutralized entropy} $\underline{h}_{\mu,d}^{r}(f)$ by
\[
  h_{\mu,d}^{r}(f):=\int h_{\mu,d}^{r}(f,x) \rd \mu, \quad \underline{h}_{\mu,d}^{r}(f):=\int \underline{h}_{\mu,d}^{r}(f,x) \rd \mu.
\]
Motivated by the notions of MME and MMHD,
a measure $\mu\in\mathcal{M}(f,X)$ is called the {\bf measure maximizing $r$-neutralized entropy},
or {\bf MM}\boldmath$r$\unboldmath{\bf NE}, if
\begin{align*}
  h_{\mu,d}^{r}(f)=\sup\left\{h_{\nu,d}^{r}(f): \nu \in \mathcal{M}(f,X)\right\}.
\end{align*}

\begin{remark}\label{metric}
  Unlike the Brin-Katok local entropy, $h_{\mu, d}^r(f,x)$ and $\underline{h}_{\mu,d}^r(f,x)$ depend on the metric $d$, 
for example, see Proposition \ref{Symbolic-proposition}. 
  However, they are independent on the Riemannian metric of smooth systems, 
since the Brin-Katok local entropy and pointwise dimensions are independent of the Riemannian metrics; 
see Theorem \ref{BPS-theorem-1} for details. 
\end{remark}

Our first result is the following entropy formula for smooth dynamical systems.

\begin{theorem}\label{LY-theorem}
Let $ f: M \rightarrow M $ be a $ C^{1+\alpha} $ diffeomorphism on a closed Riemannian manifold $ M $
and $ \mu $ an $f$-invariant Borel probability measure.
Given $r>0$, let $d^c(x)$ be the multiplicity of zero Lyapunov exponents at $x$, 
then for $\mu$-a.e. $x\in M$,
  \begin{align*}
    h^r_{\mu,d}(f,x) \leq h_\mu(f,x)+r(d^u_\mu(x)+d^s_\mu(x)+d^c(x)).
  \end{align*}
\end{theorem}


We could strengthen the entropy formula of Theorem \ref{LY-theorem} when the invariant measure $\mu$ is hyperbolic. 
Recall that an invariant measure $\mu$ is hyperbolic if all Lyapunov exponents 
are nonzero at $\mu$-almost every point. 

\begin{theorem}\label{BPS-theorem-1}
  Let $ f: M \rightarrow M $ be a $ C^{1+\alpha} $ diffeomorphism on a closed Riemannian manifold $ M $
  and $ \mu $ a hyperbolic $f$-invariant Borel probability measure.
  Given $r>0$, then for $\mu$-a.e. $x\in M$,
  \begin{align*}
    h^r_{\mu,d}(f,x)=\underline{h}^r_{\mu,d}(f,x) =h_\mu(f,x)+r(d^u_\mu(x)+d^s_\mu(x))=h_\mu(f,x)+rd_\mu(x). 
  \end{align*}
  Moreover, if $\mu$ is ergodic, then for $\mu$-a.e. $x\in M$,
  \begin{align*}
    h^r_{\mu,d}(f)=\underline{h}^r_{\mu,d}(f)=h^r_{\mu,d}(f,x)=\underline{h}^r_{\mu,d}(f,x)=h_\mu(f)+r\operatorname{dim}_H\mu.
  \end{align*}
\end{theorem}

In fact, we will prove a slightly more precise result, 
Theorem \ref{BPS-theorem}, of which Theorem \ref{BPS-theorem-1} is a corollary.
We will present relevant discussions and proofs in subsection \ref{Proof of 1.3}.

\begin{remark}\label{conjugation}
Obviously, $r$-neutralized local entropy is invariant under smooth conjugation. 
However, it is not necessarily invariant under topological conjugation. 
\end{remark}

\begin{example}\label{Anosov-example}
Let $g:(\mathbb{T}^2,d)\to (\mathbb{T}^2,d)$ 
be an Anosov diffeomorphism whose MME $\mu$ is not absolutely continuous with respect to 
 the normalized Haar measure $m$ of $\mathbb{T}^2$. 
By a well-known result of Franks \cite{Franks}, there exists 
 a linear hyperbolic diffeomorphism $T:(\mathbb{T}^2,d)\to (\mathbb{T}^2,d)$ 
 which is topologically conjugate to $g$. The MME of $(T,d,\mathbb{T}^2)$ is exactly 
 the measure $m$ and $\operatorname{dim}_H m=2$. Let $h$ be the conjugacy. We have $m=h_*\mu$.

We claim that $\operatorname{dim}_H\mu<2$. Otherwise, 
by  Ledrappier-Young  \cite{LY-85-annals-1,LY-85-annals-2}, 
$\mu$ has absolutely continuous conditional measures on unstable manifolds and stable manifolds, 
and hence $\mu$ is absolutely continuous. 
This contradicts the initial assumption.
Then we have 
\begin{align*}
  h_{\mu,d}^r(g)=h_{\mu}(g)+r\operatorname{dim}_H\mu<h_{m}(T)+r\operatorname{dim}_H m=h_{m,d}^r(T). 
\end{align*}
Therefore the $r$-neutralized entropy is not invariant under topological conjugation. 
\end{example}
 
\begin{remark}
  For a hyperbolic $f$-invariant measure $\mu$,
we have $h_{\mu,d}(f,x)=\underline{h}_{\mu,d}(f,x)$ for $\mu$-almost every point.
However, the following example shows that this may be false for $f$-invariant measures in smooth dynamical systems.
\end{remark}

\begin{example}
  According to \cite[Theorem 1]{LM85}, there are constants $a,b$ with $b>a>0$ and a $C^2$ map $f:[0,1]\to [0,1]$ preserving an $f$-invariant ergodic probability measure $\mu$
  with zero Lyapunov exponent
  such that for $\mu$-a.e. $x\in [0,1]$,
  \begin{align*}
    \varliminf_{\eta\to 0}\frac{\log \mu([x-\eta,x+\eta])}{\log \eta}=a,\quad \varlimsup_{\eta\to 0}\frac{\log \mu([x-\eta,x+\eta])}{\log \eta}=b.
  \end{align*}
  Obviously, there exist constants $C, N$ such that for $\mu$-a.e. $x\in [0,1]$ and $n\geq N$,
  \[
    B(x,Ce^{-rn-\epsilon n}) \subset B(x,n,e^{-rn}),\quad \text{ where } \epsilon=\frac{r}{3}\frac{b-a}{a}>0.
  \]
  Then for $\mu$-a.e. $x\in [0,1]$, we derive
  \begin{align*}
    \underline{h}_{\mu,d}^{r}(f,x) & = \varliminf_{n\to \infty}-\frac{1}{n}\log \mu(B(x,n,e^{-rn}))\leq \varliminf_{n\to \infty}-\frac{1}{n}\log \mu(B(x,Ce^{-rn-\epsilon n})) =a(r+\epsilon); \\
    h_{\mu,d}^{r}(f,x)             & =\varlimsup_{n\to \infty}-\frac{1}{n}\log \mu(B(x,n,e^{-rn}))  \geq \varlimsup_{n\to \infty}-\frac{1}{n}\log \mu(B(x,e^{-rn}))=br>a(r+\epsilon).
  \end{align*}
  Therefore we conclude that $h_{\mu,d}^{r}(f,x)\neq \underline{h}_{\mu,d}^{r}(f,x)$ for $\mu$-a.e. $x\in [0,1]$.
\end{example}

The following is an immediate corollary of the previous theorems. 

\begin{corollary}
  Let $ f: M \rightarrow M $ be a $ C^{1+\alpha} $ diffeomorphism on a closed Riemannian manifold $ M $
and $ \mu $ an $f$-invariant Borel probability measure.
Given $r>0$, then 
\begin{align*}
  h^r_{\mu,d}(f) \leq h_\mu(f)+r\int_M \left( d^u_\mu(x)+d^s_\mu(x)+d^c(x)\right) \,\rd \mu. 
\end{align*}
Moreover, if $\mu$ is hyperbolic, then 
\begin{align*}
  h^r_{\mu,d}(f)= h_\mu(f)+r\int_M \left( d^u_\mu(x)+d^s_\mu(x) \right) \,\rd \mu=h_\mu(f)+r\int_M d_\mu(x) \,\rd \mu.
\end{align*}
\end{corollary}

The above  results are contributions to our understanding of the notion of metric entropy in smooth systems.
On one hand, metric entropy measures the exponential rates of growth of $n$-orbits from an ergodic perspective. 
Our work shows the dependence of metric entropy on $\epsilon_n$, where $\epsilon_n$ 
in the Bowen ball $B(x,n,\epsilon_n)$ depends on $n$ at an appropriate scale. 
On the other hand, 
the $r$-neutralized entropy combines the global complexity, as measured by metric entropy, with the local fractal property, given by pointwise dimensions of the measure.

\subsection{Extremal measures}

Measures with certain special properties, which we refer to as {\bf extremal measures},
 play a crucial role in dynamical systems.
Equilibrium states constitute a large class of extremal measures. 
Particularly, the MME mentioned earlier is a special case, in the sense that the potential function $\phi$ is constant 
(see subsection \ref{Pressure and some facts} for more details). 

The study of extremal measures has a rich history. 
 In thermodynamic formalism, 
 a key question is determining the conditions under which 
 a dynamical system $(X,f,\phi)$ has a unique equilibrium measure. 
Bowen established the existence and uniqueness of the equilibrium state for uniformly hyperbolic systems with Hölder continuous potentials \cite{Bo75}. 
Following Bowen's foundational work, subsequent research has partially extended these results
(mostly on SRB measures and MME) from uniformly hyperbolic systems 
to various contexts, including
geodesic flows \cite{BCFT18, CKW20, Kn98}, 
non-uniformly hyperbolic systems \cite{CT-advance, Wan}, and partially hyperbolic systems \cite{CPZ-JMD, RHRHTU12}, among others.

In contrast, 
 the study of MMHD appears to be more challenging, 
even for uniformly hyperbolic systems. 
For hyperbolic sets of surface diffeomorphisms, Barreira and Wolf established the existence of ergodic MMHD in \cite{BW03}. 
We emphasize that MMHD may not be unique, even in the case of linear horseshoes\cite{Rams}. 
For further information on pointwise dimensions in smooth dynamics, we refer the reader to \cite{BW03,BW06,BG11, FCZ20}.

Our next result establishes the existence of ergodic MM$r$NE in the case of two-dimensional uniformly hyperbolic systems, for any positive $r$. 
Throughout this paper, the space of measures is equipped with the weak* topology.

\begin{theorem}\label{CMP}
Let $ f $ be a $ C^{1+\alpha} $ surface diffeomorphism, and let $ \Lambda $ be a compact locally maximal hyperbolic set of $ f$  such that $ f|_\Lambda $ is topologically mixing. 
 Given $r>0$, there exists an ergodic measure $\mu_r$ maximizing $r$-neutralized  entropy  on $\Lambda$. 

Moreover, 
suppose $r_n$ is a sequence of positive real numbers,
\begin{itemize}
\item[(1)]   $\mu_{r_n}$ converge to the measure of maximal entropy of $(f|_\Lambda, \Lambda)$ if $r_n\to 0$; 
\item[(2)] any limit point of $\mu_{r_n}$  is the measure of maximal Hausdorff dimension if $r_n\to \infty$. 
\end{itemize}
\end{theorem}

\begin{remark}\label{non-equilibrium}
It is possible that, the ergodic measure $\mu_r$ obtained in Theorem \ref{CMP} is NOT an equilibrium state. See subsection \ref{sec:pro} for more details.
\end{remark}

\begin{remark}
By applying a method in \cite{Rams},
we show the existence of a two-dimensional linear horseshoe that has 
at least two MM$r$NEs for any $r\ge 3$, 
see subsection \ref{multiple subsection}, Theorem \ref{multiple} for more details. 
\end{remark}

The following result is a consequence of Theorem \ref{BPS-theorem-1} and Theorem \ref{CMP}. 

\begin{corollary}\label{corollary}
Let $ f$  be a $ C^{1+\alpha} $ surface diffeomorphism, and let $ \Lambda $ be a compact locally maximal hyperbolic set of $ f$  such that $ f|_\Lambda $ is topologically mixing.
Given $r>0$, there exists an ergodic measure $\mu_r$ on $\Lambda$ such that
  \begin{align*}
    h_{\mu_r}(f)+r\operatorname{dim}_H\mu_r=\sup\{h_{\nu}(f)+r\operatorname{dim}_H\nu: \nu \text{ is an ergodic measure on } \Lambda\}. 
  \end{align*}
\end{corollary}
The $r$-neutralized entropy provides a balance between the metric entropy and
the integral of the pointwise dimension. 
By characterizing the limit of the measures $\mu_r$ as $ r $ approaches $0$ or infinity, 
we establish the relation among MME, MMHD, and MM$r$NE. 
In subsection \ref{futher discussion}, we present some rigidity results related to MM$r$NE. 

In general, the existence of ergodic MM$r$NE is an extremely challenging problem.
For smooth systems exhibiting some hyperoblicity, the study of
 Hausdorff dimension of invariant measures is generally less studied 
 (and even much more difficult) compared to that of (metric/topological) entropy. There are various reasons.
For instance, the map $\nu\mapsto h_\nu(f)$ is upper-semicontinuous, while
the map $\nu\mapsto \operatorname{dim}_H \nu$ is neither upper-semicontinuous 
nor lower-semicontinuous. 
Therefore, when considering ergodic MM$r$NE, it is particularly important to take into account the impact of the Hausdorff dimension.

On one hand, it is possible to construct an example that does not have 
ergodic MM$r$NE for any positive $r$. In \cite{UW08}, Urb\'anski and Wolf considered a two-dimensional horseshoe map that is uniformly hyperbolic except at a parabolic point. 
It is proven that this horseshoe map has no ergodic MMHD. 
We can employ a similar method to demonstrate that this kind of map does not possess an ergodic MM$r$NE for any positive $r$.
This indicates that systems with an ergodic MM$r$NE shall exhibit strong hyperbolicity at every point. 

On the other hand, even with full hyperbolicity ( i.e., Anosov systems), 
there are technical difficulties in proving the existence of ergodic MM$r$NE 
in high-dimensional case, due to the complicated structures of both the stable and unstable manifolds.


\subsection{Outline of the paper}
In \S $2$,  we present some preliminaries  in smooth systems and complete the proofs of Theorem \ref{LY-theorem} and  Theorem \ref{BPS-theorem-1}. 
In \S \ref{sec:pro}, we prove Theorem \ref{CMP}, 
as well as Corollary \ref{corollary}. 
Additionally in \S \ref{multiple subsection}, we construct a linear horseshoe having at least two ergodic MM$r$NEs (Theorem \ref{multiple}). 
Furthermore, in \S \ref{futher discussion} we include some rigidity results.
In the Appendix, we provide some auxiliary results and discussions on variational principles for $r$-neutralized entropy. 

\medskip

{\bf Acknowledgement:} C. Dong was supported by Nankai Zhide Foundation and “the Fundamental Research Funds for the Central Universities” No. 100-63233106 and 100-63243066. 
Q. Qiao was supported by Nankai Zhide Foundation.

\section{$r$-neutralized local entropy for smooth systems}\label{differentiable dynamics}

Let $f$ be a $ C^{1+\alpha}$ diffeomorphism (with $\alpha>0$) on a closed Riemannian manifold $M$ and  $\mu $ an $ f$-invariant Borel probability measure on $ M $. 
We use $d$ to denote the Riemannian metric on $M$. 
A point $x$ is called regular if there exist numbers $\lambda_1(x)> \cdots>\lambda_{r(x)}(x)$ and a decomposition of the tangent space $T_x M$ at $x$ into $T_x M=E_1(x) \oplus \cdots \oplus E_{r(x)}(x)$ such that for every tangent vector $v \in E_i(x)$ with $v \neq 0$,
\begin{align}\label{LY-0}
  \lim _{n \rightarrow \pm \infty} \frac{1}{n} \log \left\|D f_x^n v\right\|=\lambda_i(x).
\end{align}
The numbers $\lambda_i(x), i=1, \ldots, r(x)$, are called the Lyapunov exponents of $f$ at $x$ and $\operatorname{dim} E_i(x)$ is called the multiplicity of $\lambda_i(x)$. 
According to the Multiplicative Ergodic Theorem \cite{Oseledec}, 
the set $\Gamma^{\prime}$ of regular points has full measure, 
and the functions $ r(x), \lambda_i(x)$ and 
$\operatorname{dim} E_i(x)$ are measurable and $ f $-invariant.

For $x\in \Gamma'$, we set $\lambda^+(x)=\min\{\lambda_i, \lambda_i>0\}$, $\lambda^-(x)=\max\{\lambda_i, \lambda_i<0\}$ and
\begin{align*}
    E^s(x)  =\underset{\lambda_i(x)<0}{\bigoplus} E_i(x), \quad E^u(x)  =\underset{\lambda_i(x)>0}{\bigoplus} E_i(x), \quad E^c(x) & =E_{i_0}(x) \text{ with } \lambda_{i_0}(x)=0.
   \end{align*}
There exists a measurable function $ r(x)>0 $ such that for $ \mu $-a.e. $ x \in M $, 
the stable and unstable local manifolds at $ x $, defined as follows, 
are immersed local manifolds. 
\begin{align*}
  W^{s}(x)&=\left\{y \in B(x, r(x)): \varlimsup\limits_{n \rightarrow +\infty} \frac{1}{n} \log d\left(f^{n} x, f^{n} y\right)<0\right\}, \\
  W^{u}(x)&=\left\{y \in B(x, r(x)): \varliminf\limits_{n \rightarrow -\infty} \frac{1}{n} \log d\left(f^{n} x, f^{n} y\right)>0\right\}. 
\end{align*}
For any $0<r<r(x) $, the set $ B^{s}(x, r) \subset W^{s}(x) $ is the ball 
centered at $ x $
with respect to the induced distances on $ W^{s}(x) $. 
Similarly, we can define the set $ B^{u}(x, r) \subset   W^{u}(x) $. 

Next, let us recall some facts about Lyapunov charts (see \cite{BPS-99-annals,LY-85-annals-1,LY-85-annals-2,BP-nonuniform} for more details).
 For any $  (x, y, z) \in \mathbb{R}^u \times \mathbb{R}^c \times \mathbb{R}^s$,
we define 
\[
|(x, y, z)|=\max \left\{|x|_u,|y|_c,|z|_s\right\}
\]
where $|\cdot|_u,|\cdot|_c$ and $|\cdot|_s$ are the Euclidean norms on $\mathbb{R}^u, \mathbb{R}^c$ and $\mathbb{R}^s$ respectively.

For $x\in \Gamma'$ and $0<\epsilon<\min\{\frac{1}{100}\lambda^{+}(x),-\frac{1}{100}\lambda^{-}(x)\}$, 
there exists a measurable function $l: \Gamma^{\prime} \rightarrow[1, \infty)$ with $l\left(f^{ \pm} x\right) \leq$ $e^{\epsilon} l(x)$
and an embedding $\Phi_x: R\left(l(x)^{-1}\right) \rightarrow M$ 
such that 
\begin{itemize}
  \item[(i)] $\Phi_x 0=x $, $D \Phi_x(0)$ takes $\mathbb{R}^u, \mathbb{R}^c$ and $\mathbb{R}^s$ to $E^u(x), E^c(x)$ and $E^s(x)$ respectively.

  \item[(ii)] 
We define the connecting map $\tilde{f}_x=\Phi_{f(x)}^{-1} \circ f \circ \Phi_x$ whenever it make sense. 
Similarly, we define the function $\tilde{f}_x^{-1}=$ $\Phi_{f^{-1} x}^{-1} \circ f^{-1} \circ \Phi_x$. 
Then  
\begin{align*}
        e^{\lambda^{+}-\epsilon}|v| & \leq\left|D \tilde{f}_x(0) v\right| ,       &&             \text { for } v \in \mathbb{R}^u; \\
        e^{-\epsilon}|v|           & \leq\left|D \tilde{f}_x(0) v\right| \leq e^{\epsilon}|v|,  &&\text { for } v \in \mathbb{R}^c; \\
      \left|D \tilde{f}_x(0) v\right|& \leq e^{\lambda^{-}+\epsilon}|v|,  && \text { for } v \in \mathbb{R}^s .
\end{align*}

  \item[(iii)] If $L(g)$ denotes the Lipschitz constant of the function $g$, then
    \begin{align*}
      L\left(\tilde{f}_x-D \tilde{f}_x(0)\right) \leq \epsilon,\quad L\left(\tilde{f}_x^{-1}-D \tilde{f}_x^{-1}(0)\right) \leq \epsilon,\quad   L\left(D \tilde{f}_x\right), L\left(D \tilde{f}_x^{-1}\right) \leq l(x) .
    \end{align*}

  \item[(iv)] For any $z, z^{\prime} \in R\left(l(x)^{-1}\right)$, there exists a constant $K$ such that
    $$
      K^{-1} d\left(\Phi_x z, \Phi_x z^{\prime}\right) \leq\left|z-z^{\prime}\right| \leq l(x) d\left(\Phi_x z, \Phi_x z^{\prime}\right).
    $$
\end{itemize}

In \cite{LY-85-annals-1, LY-85-annals-2}, Ledrappier and Young constructed measurable partitions $ \xi^{s} $ and $\xi^{u}$ of $ M $ such that
for $ \mu $-a.e. $ x \in M $, $ \xi^{s}(x) \subset W^{s}(x) $ and $ \xi^{u}(x) \subset W^{u}(x) $. 
Furthermore, $ \xi^{s}(x) $ and $ \xi^{u}(x) $ contain the intersection of an open neighborhood 
of $ x $ with $ W^{s}(x) $ and $ W^{u}(x) $, respectively.
We denote by $ \mu_{x}^{s} $ and $ \mu_{x}^{u} $ the conditional measures of $ \mu $ 
with respect to the partitions $ \xi^{s} $ and $ \xi^{u} $, respectively. 

For $\mu\in \mathcal{M}(f,M)$ and $x\in M$, 
the stable pointwise dimensions $d^{s}_\mu(x)$ and 
unstable pointwise dimensions $d^{u}_\mu(x)$ are defined by 
\begin{align*}
  d^{s}_\mu(x) := \lim _{r \rightarrow 0} \frac{\log \mu_{x}^{s}\left(B^{s}(x, r)\right)}{\log r}, \quad d^{u}_\mu(x) := \lim _{r \rightarrow 0} \frac{\log \mu_{x}^{s}\left(B^{u}(x, r)\right)}{\log r}. 
\end{align*}
In \cite{LY-85-annals-2}, Ledrappier and Young proved the above limits exist for $\mu$-a.e. $x\in M$. 

In \cite{BPS-99-annals}, Barreira, Pesin, and Schmeling 
proved the following theorem, 
which existence of the pointwise dimension $d_\mu(s)$ for $\mu$-a.e. $x\in M$ 
when $\mu$ is hyperbolic. 

\begin{theorem}[\cite{BPS-99-annals}]
Let $ f: M \rightarrow M $ be a $ C^{1+\alpha} $ diffeomorphism on a closed Riemannian manifold $ M $
and $ \mu $ a hyperbolic $f$-invariant Borel probability measure. Then for $\mu$-a.e. $x\in M$, 
$d_\mu(x)=d_\mu^u(x)+d_\mu^s(x)$. 
\end{theorem}

Building upon the notion of leaf pointwise dimensions in \cite{LY-85-annals-2}, 
Ben Ovadia developed a significant entropy formula for the $r$-neutralized Bowen ball 
on the unstable manifold. 
This formula provides a key tool for our investigation of $r$-neutralized local entropy. 

\begin{theorem}[\cite{O24}]\label{O-u}
  Let $f$ be a $ C^{1+\alpha} $ diffeomorphism on a closed Riemannian manifold $M$ and  $\mu $ an $ f$-invariant Borel probability measure on $ M $, given $r>0$, then
  \begin{align*}
    \lim_{n\to\infty}-\frac{1}{n}\log \mu^u_x(B^u(x,n,e^{-nr}))=h_\mu(f,x)+rd^u(x).
  \end{align*}
\end{theorem}

\begin{remark}
While the author initially proved this theorem for ergodic measures, 
the argument can be naturally extended to nonergodic $f$-invariant measures. 
The formula involves local quantities such as Brin-Katok local entropy and unstable pointwise dimension,
which are defined pointwise and do not require ergodicity. 
Thus, the proof technique applies directly to the nonergodic case without modification. 
\end{remark}

Furthermore, Ben Ovadia established the asymptotic local product structure for conditional measures
on intermediate foliations of unstable leaves in smooth systems\cite{O24}. 

We are now ready to prove Theorem \ref{LY-theorem} and Theorem \ref{BPS-theorem-1}. 
Throughout this section, we fix the parameter $r>0$.

\subsection{Proof of Theorem \ref{LY-theorem}}\label{partition}

  Motivated by \cite[Theorem F]{LY-85-annals-2},
  our strategy of proving Theorem \ref{LY-theorem} is to use the properties of special partitions. 
We begin by introducing some notions used in \cite{LY-85-annals-1, LY-85-annals-2}.

Fix $\epsilon>0$. For integers $u_0, s_0$, real numbers $d^u, d^s$, $h$, $\lambda^{+}$ 
and $\lambda^{-}$ with $\lambda^{+},-\lambda^{-}>100 \epsilon$, 
we denote by $\Gamma\left(\epsilon, u_0, s_0, d^u, d^s, h, \lambda^{+}, \lambda^{-}\right)$ 
the set of the point $x \in \Gamma^{\prime}$ satisfying the following properties: 
\begin{itemize}
  \item[(i)] $\operatorname{dim} E^u(x)=u_0$, $\operatorname{dim} E^s(x)=s_0$;

  \item[(ii)] $\min\limits_{\lambda_i(x)>0} \lambda_i(x) \geq \lambda^{+}$, $\max\limits_{\lambda_i(x)<0} \lambda_i(x) \leq \lambda^{-}$;

  \item[(iii)] $d^u-\epsilon \leq d^u(x) \leq d^u$, $d^s-\epsilon \leq d^s(x) \leq d^s$;

  \item[(iv)] $h \leq h_\mu(f,x) \leq h+\epsilon$.
\end{itemize}
Without loss of generality, we assume that for some 
$\epsilon$, $u_0$, $s_0$, $\lambda^{+}$, $\lambda^{-}, d^u, d^s$ and $h$, 
\[
\mu\left(\Gamma(\epsilon, u_0, s_0, d^u, d^s, h, \lambda^{+}, \lambda^{-})\right)=1.
\]
For the splitting $T_xM=E^u(x) \oplus E^c(x) \oplus E^s(x)$, 
we use $\left(v_x^u, v_x^c, v_x^s\right)$ to represent the coordinates of the 
tangent space $T_x M$. 
By utilizing the exponential mapping at $x$, 
this establishes a coordinate system in a neighborhood of $x$. 
According to Lusin's theorem, there exists a compact set $\Lambda$ with 
$\mu( \Lambda)>1-\frac{1}{3}\epsilon$ such that the functions 
$E^u(x), E^c(x), E^s(x)$ and $l(x)$ are continuous on $\Lambda$. 
Let $L=\max \left\{l(x), x \in \Lambda\right\}$. 
There exists $\delta>0$ such that if $d\left(x, x^{\prime}\right)<\delta$, then 
\begin{itemize}
\item[(i)] $\exp _{x^{\prime}}^{-1} \exp _x(v)$ is defined whenever $\|v\| \leq 3 \delta$, 

\item[(ii)] The slope of $\exp _{x^{\prime}}^{-1} \exp _x\left\{v_x^\alpha=\text{ constant } \right\}$ is less than $(4KL)^{-1}$ relative to $E^\alpha\left(x^{\prime}\right)$ for $\alpha=u$, $c$ and $s$. 
\end{itemize}
Let $\mathcal{L}_0$ be a partition of $\Lambda$ into the sets of diameter smaller than $\delta$.  
For every $q \in \mathcal{L}_0$, we choose a point $z(q) \in q$. 
For any $x \in \Lambda$, we write 
\begin{align*}
\bar{x}=\bar{x}^u+\bar{x}^c+\bar{x}^s, \quad \text{ where } \bar{x}^\alpha=\left(\exp _{z(q(x))}^{-1} x\right)_{z(q(x))}^\alpha, \alpha=u, c, s. 
\end{align*}
Clearly, we could assume that 
$2^{-1}d\left(x, x^{\prime}\right) \leq$ $\left|\bar{x}-\bar{x}^{\prime}\right| \leq 2 d\left(x, x^{\prime}\right)$ for any $x^{\prime} \in q(x)$. 

The following lemma will play an important role in the proof of Theorem \ref{LY-theorem}. 

\begin{lemma}[\cite{LY-85-annals-2}]\label{LY-Sublemma 12.2.2} 
Let $\delta$ be sufficiently small. For $x \in \Lambda$, we define the set
$$
  V(x, n)=\left\{y \in q(x): d\left(f^j x, f^j y\right) \leq \frac{\delta}{2} l\left(f^j x\right)^{-2},\quad -n b \leq j \leq n a\right\} .
$$
If $y_1, y_2 \in V(x, n)$ satisfy $\left|\bar{y}_1^c-\bar{y}_2^c\right| \leq e^{-n}(12 K L)^{-1}$, then $\left|\bar{y}_1-\bar{y}_2\right| \leq e^{-n}$.
\end{lemma}

For integers $ k, l \geq 1 $ and partition $ \eta $, 
we define the partition $ \eta_{k}^{l}=\bigvee_{n=-k}^{l} f^{-n} \eta $.  
Naturally, we have $ \eta_{k}^{0}(x) \cap \eta_{0}^{l}(x)=\eta_{k}^{l}(x) $. 
{\bf Throughout this section, without confusions, numbers like $anr$ and $bnr$ of the partition $\eta_{bnr}^{anr}$ shall be considered as 
 the corresponding integer parts of original values.}
 
Following the proof of\cite[Lemma 12.2.1]{LY-85-annals-2}, 
we prove the following lemma. 
\begin{lemma}\label{LY-Lemma 12.2.1}
Let $a=(\lambda^{+})^{-1}-2 \epsilon, b=-(\lambda^{-})^{-1}-2 \epsilon$. There exist $a$ set $\Lambda_2$ with $m\left(\Lambda_2\right)>1-\epsilon$, an integer $N_0$, a constant $C$ and a finite entropy partition $\mathcal{P}$ 
with the following property: for every $n \geq N_0$, there exists a constant $C$ and a partition $\mathcal{L}_n$ 
refining $\mathcal{P}_{nbr+n}^{n ar+n}$ such that 
if $x \in \Lambda_2$ and $q_n(x)$ is the atom of $\mathcal{L}_n$ containing $x$, then:
\begin{itemize}
\item[(i)] $\operatorname{diam} q_n(x) \leq 2 e^{-nr}$ and $q_n(x)\subset B(x,n,2e^{-nr})$,

\item[(ii)] $\mu( q_n(x)) \geq C^{-1} \mu( \mathcal{P}_{nbr+n}^{n ar+n}(x)) e^{-n \epsilon} e^{-n(r+\epsilon) d^c}$.
\end{itemize}
\end{lemma}
\begin{proof}
The partition $\mathcal{P}$ we used 
is identical to the construction of \cite[Lemma 12.2.1]{LY-85-annals-2}. 
Hence we only need to show the derivation of the partition $\mathcal{L}_n$ 
from the partition $\mathcal{P}$. 

Let $\mathcal{L}_n|_{\mathcal{P}_{nbr+n}^{n ar+n}(x)}$ be a partition such that 
if $y_1$ and $y_2$ are in the same element of $\mathcal{L}_n$, then $\left|\bar{y}_1^c-\bar{y}_2^c\right| \leq e^{-nr-n\epsilon} (12 K L)^{-1}$. 
Fix an atom $q_n(z)$ of $\mathcal{L}_n$, 
since the Lyapunov exponents along the central direction are zero, 
there exists a constant $N_0>0$ such that for any points $y_1,y_2$ in $q_n(z)$ and $n\geq N_0$,
\[
\max_{0\leq j\leq n-1}\left|\overline{f^jy}_1^c-\overline{f^jy}_2^c\right| \leq \max_{0\leq j\leq n-1} e^{j\epsilon}e^{-nr-n\epsilon} (12 K L)^{-1}\leq e^{-nr} (12 K L)^{-1}.
\]
  According to Lemma \ref{LY-Sublemma 12.2.2}, 
we conclude that $ \operatorname{diam} q_n(f^j x)\leq 2e^{-nr} $ for $ 0 \leq j \leq n-1 $. 
This implies $q_n(x)\subset B(x,n,2e^{-nr})$. 
Obviously, under proper arrangement, 
the cardinality of $\left.\mathcal{L}_n\right|_{\mathcal{P}_{nbr+n}^{nar+n}(x)}$ could be less than $(24 K L)^{d^c}e^{n(r+\epsilon) d^c}$. 
Let us introduce the set below. 
$$
  A_n=\left\{x: \mu\left(q_n(x) \cap \mathcal{P}_{nbr+n}^{nar+n}(x)\right) \leq \frac{e^{-n(r+\epsilon) d^c}}{(24 K L)^{d^c}} \frac{\epsilon^{\prime}}{3} e^{-n \epsilon}\left(1-e^{-\epsilon}\right) \mu\left( \mathcal{P}_{nbr+n}^{n ar+n}(x)\right)\right\} .
$$
Therefore we have
\begin{align*}
\mu( A_n) \leq \frac{e^{-n(r+\epsilon) d^c}}{(24 K L)^{d^c}} \frac{\epsilon^{\prime}}{3} e^{-n \epsilon}\left(1-e^{-\epsilon}\right) \cdot (24 K L)^{d^c} e^{n(r+\epsilon) d^c}= \frac{\epsilon^{\prime}}{3} e^{-n \epsilon}\left(1-e^{-\epsilon}\right).
\end{align*}
 By removing the set $A_n$ of small positive measure and choosing an appropriate constant $C$, 
we complete the proof of this lemma.
\end{proof}

We present a slightly modified version of the Bowen-Lebesgue density lemma\cite{ORH23} for our purposes here. 

\begin{lemma}[Bowen-Lebesgue density lemma\cite{ORH23}]\label{Proposition 3}
  Let $f$ be a $ C^{1+\alpha} $ diffeomorphism on a closed Riemannian manifold $M$ and  $\mu $ an $ f$-invariant Borel probability measure on $ M $. 
Let $ A $ be a measurable set such that $\mu(A)>0 $.
  Then for $ \mu $-a.e $ x \in A $,
  \[
    \lim_{n \rightarrow \infty} -\frac{1}{n} \log \frac{\mu\left(B\left(x, n, e^{-nr}\right) \cap A\right)}{\mu\left(B\left(x, n, e^{-nr}\right)\right)}=0.
  \]
\end{lemma}
\begin{remark}
      In fact, the original version of the lemma is established under the limit $r\to 0$. 
Since the main tool \cite[Lemma 2.2]{ORH23} is still valid for any fixed positive $r$, 
the original lemma can be extended to positive $r$ naturally. 
\end{remark}

Now, we proceed with the main proof of Theorem \ref{LY-theorem}. 

Based on previous discussions, given $ 0<\epsilon<1 $, there exists a set $ \Gamma \subset M $ of measure $ \mu(\Gamma)>1-\frac{\epsilon}{4} $, an integer $n_{0} \geq 1 $,
and a constant $  C>1 $ such that for every $ x \in \Gamma $ and any integer $ n \geq n_{0} $, 
the following statements hold
(if there is no confusion, we will use $h, d^c, d^u, d^s$ for simplicity instead of $h_\mu(f,x), d^c(x), d^u(x), d^s(x)$ respectively): 
\begin{itemize}
  \item[(a)] By Shannon-McMillan-Bremian theorem, for all integers $ k, l \geq 1 $ we have
    \begin{align}
      C^{-1} e^{-(l+k) (h+\epsilon)} & \leq \mu\left(\mathcal{P}_{k}^{l}(x)\right) \leq C e^{-(l+k) (h-\epsilon)}, \label{LY-(a)} \\
      C^{-1} e^{-k h-k \epsilon}         & \leq \mu_{x}^{s}\left(\mathcal{P}_{k}^{0}(x)\right) \leq C e^{-k h+k \epsilon}, \label{LY-(b)} \\
      C^{-1} e^{-l h-l \epsilon}         & \leq \mu_{x}^{u}\left(\mathcal{P}_{0}^{l}(x)\right) \leq C e^{-l h+l \epsilon}. \label{LY-(c)}
    \end{align}
    \item[(b)]
\begin{align}
  \xi^{s}(x) \cap \bigcap_{n \geq 0} \mathcal{P}_{0}^{n}(x) \supset B^{s}\left(x, e^{-n_{0}r}\right),  \label{P-(8)}\\ 
  \xi^{u}(x) \cap \bigcap_{n \geq 0} \mathcal{P}_{n}^{0}(x) \supset B^{u}\left(x, e^{-n_{0}r}\right). \label{P-(9)}
\end{align}
    \item[(c)] Based on lemma \ref{LY-Lemma 12.2.1}, for every $n\geq n_0$, there exists a partition $\mathcal{L}_n$ refining $\mathcal{P}^{anr+n}_{bnr+n}$ such that if $q_n(x)$ is the atom of $\mathcal{L}_n$ containing $x$, then
    \begin{align}
      q_n(x) &\subset B(x,n, 2e^{-nr}) \label{LY-(d)} \\
      \mu(q_n(x))&\geq C^{-1}e^{-n\epsilon}e^{-(nr+n\epsilon)d^c}e^{-n(ar+br+2)(h+2\epsilon)}. \label{LY-(e)}
    \end{align}
  \item[(d)] By the definition of the stable pointwise dimension $d^s$ and Theorem \ref{O-u}, we have
    \begin{align}
      e^{-d^{s} rn-nr \epsilon}               & \leq \mu_{x}^{s}\left(B^{s}\left(x, e^{-nr}\right)\right) \leq e^{-d^{s} nr+nr \epsilon},       \label{P-(10)}           \\
      e^{-(rd^{u}+h)n-nr \epsilon} & \leq \mu_{x}^{u}\left(B^{u}\left(x, n, e^{-nr}\right)\right) \leq e^{-(rd^{u}+h)n+nr \epsilon}. \label{P-(11)}
    \end{align}
    \item[(e)] According to Lemma \ref{Proposition 3}, we obtain
\begin{align}
  \mu_x^u\left(\Gamma \cap B^u\left(x, n, e^{-nr}\right)\right) & \geq e^{-n\epsilon}\mu_x^u\left( B^u\left(x, n, e^{-nr}\right)\right), \label{LY-(f)} \\
 \mu_x^s(B(x, n, e^{-nr}) \cap \Gamma)                             & \geq e^{-n\epsilon} \mu_x^s(B(x, n, e^{-nr})), \label{LY-f-s}         \\
  \mu(B(x, n, e^{-nr}) \cap \Gamma)                             & \geq e^{-n\epsilon} \mu(B(x, n, e^{-nr})), \label{P-(19)}   \\
  \mu\left(B(x, n, 4 e^{-nr}) \right)                           & \leq e^{-n(h^r_{\mu,d}(f)-\epsilon)} .\label{LY-(h)}
\end{align}
\end{itemize}
  Given $n\in \mathbb{N}$ with $n \geq n_0$ and $4 C \leq e^{n \epsilon}$, 
  we consider the following number,
$$
  N_n=\#\left\{\text { atoms of } q_n \text { intersecting } \Gamma \cap B\left(x, n, 2 e^{-nr}\right)\right\}.
$$
By combining \eqref{LY-(d)}, \eqref{LY-(e)} and \eqref{LY-(h)}, we obtain
\begin{align}\label{LY-(*)}
  N_n \leq C e^{n \epsilon} e^{n(r+\epsilon) d^c} e^{n(ar+br+2)(h+2 \epsilon)} e^{-n(h_{\mu,d}^r-\epsilon)} .
\end{align}

On the other hand, for $x \in \Gamma$ and $y\in \xi^s(x) \cap \Gamma \cap B\left(x, e^{-n}\right)$,  the property \eqref{LY-(b)} implies 
$$
  \mu_x^s\left(\mathcal{P}_{nr b+n}^0(y)\right)=\mu_y^s\left(\mathcal{P}_{nr b+n}^0(y)\right) \leq Ce^{-n(br+1)(h- \epsilon)} .
$$
Recall that the Lyapunov exponents along the central direction are zero, 
we assume $B(x,n,e^{-nr})\subset B(x,e^{-nr-n\epsilon})$ for any $n\geq n_0$
(increasing $ n_{0} $ if necessary). 
Let $N^s_n$ be the cardinality of the atoms of the partition $\mathcal{P}_{nrb+n}^0 $ 
intersecting the set $\Gamma \cap B\left(x, n, e^{-nr}\right)$. 
Combining \eqref{LY-(b)}, \eqref{P-(10)} and \eqref{LY-f-s}, we obtain
\begin{align}
N^s_n &\geq \frac{\mu^s_x\left(\Gamma\cap B(x,n,e^{-nr})\right)}{Ce^{-n(br+1)(h-\epsilon)}}\geq C^{-1}\mu^s_x( B(x,n,e^{-nr}))e^{-n\epsilon+n(br+1)(h-\epsilon)}\nonumber\\
&\geq C^{-1}\mu^s_x( B(x,e^{-nr-n\epsilon}))e^{-n\epsilon+n(br+1)(h-\epsilon)}\geq C^{-1} e^{-d^snr-nr\epsilon}e^{-n\epsilon+n(br+1)(h-\epsilon)}.\label{LY-1-1}
\end{align}
For a fixed atom $p_s$ counted in $N^s_n$ and $y \in p_s \cap \Gamma \cap B\left(x, n,e^{-nr}\right)$, 
the property \eqref{LY-(c)} implies that
 for any $z\in \xi^u(y)\cap \mathcal{P}_{-\infty}^0(y) \cap \Gamma \cap B\left(y,n, e^{-nr}\right)$, 
\begin{align}\label{LY-2-0}
  \mu^u_y\left(\mathcal{P}_0^{nra+n}(z)\right)=\mu^u_z\left(\mathcal{P}_0^{nra+n}(z)\right) \leq Ce^{-n(ar+1)(h-\epsilon)} .
\end{align}
For a measurable subset $X$, we use $n(X)$ to denote the number of atoms in $\mathcal{P}_0^{nra+n}$ 
 intersecting $X \cap \Gamma \cap B\left(y, n, e^{-nr}\right)$. By applying \eqref{P-(8)}, \eqref{P-(11)}, \eqref{LY-(f)} and \eqref{LY-2-0}, we get
\begin{align}\label{LY-2-1}
  n\left(\xi^u(y)\cap \mathcal{P}_{-\infty}^0\right) \geq \frac{\mu^u_y\left( \xi^u(y)\cap \Gamma\cap B(y,n,e^{-nr}) \right)}{Ce^{-n(ar+1)(h-\epsilon)}}
  \geq C^{-1}e^{-n\epsilon} e^{-n(h+rd^u+\epsilon r)} e^{n (ar+1)(h-\epsilon)} .
\end{align}
For any atom $p_u$ in this count, $p_u \cap p_s$ is an atom of the partition $\mathcal{P}_{n rb+n}^{nr a+n}$ 
which intersects $\Gamma \cap B\left(y,n, e^{-nr}\right)$ for some $y$ with $\max\limits_{0\leq i\leq n-1}d(f^i(y), f^i(x)) \leq e^{-nr}$. 
Since $\mathcal{P}_{-\infty}^0(y) \subset p_u$, 
we have $  n\left(p_u\right) \geq n\left(\xi^u(y)\cap \mathcal{P}_{-\infty}^0\right) $. 
Applying \eqref{LY-1-1} and \eqref{LY-2-1}, we have
\begin{align}\label{LY-(**)}
  N  \geq \sum_{\left\{p_u:\, p_u \cap \Gamma \cap B\left(x,n, e^{-nr}\right) \neq \varnothing\right\}}n(p_u)
\geq  N^s_n \cdot C^{-1}e^{-n\epsilon} e^{-n(h+rd^u+\epsilon r)} e^{n (ar+1)(h-\epsilon)}. 
\end{align}
The proof of Theorem \ref{LY-theorem} follows from \eqref{LY-(*)}, \eqref{LY-(**)} and the arbitrariness of $\epsilon$.

\subsection{Proof of Theorem \ref{BPS-theorem-1}}\label{Proof of 1.3}

  In this subsection, we will prove Theorem \ref{BPS-theorem} in the following, 
which implies Theorem \ref{BPS-theorem-1} as previously mentioned.

\begin{theorem}\label{BPS-theorem}
  Let $ f: M \rightarrow M $ be a $ C^{1+\alpha} $ diffeomorphism on a closed Riemannian manifold $ M $
  and $ \mu $ a hyperbolic $f$-invariant Borel probability measure.
  Given $r>0$, for any $ \epsilon>0 $, there exist a subset $ \Lambda \subset M $ with $ \mu(\Lambda)>1-\epsilon $
  and constants  $b=b(r)$, $t=t(r)$, $N\in \mathbb{N}$ such that for every $ y \in \Lambda $ and $n\geq N$,
  \begin{align*}
    \mu_{y}^{s}\left(B^{s}(y, e^{-nr})\right) \mu_{y}^{u}\left(B^{u}(y, n,e^{-nr})\right) \leq \mu\left(B(y, n, 4 e^{-nr})\right) e^{tn \epsilon}, \\
    \mu\left(B(y, n+b,e^{-r(n+b)})\right) \leq \mu_{y}^{s}\left(B^{s}(y, 4 e^{-nr})\right) \mu_{y}^{u}\left(B^{u}(y, n,4 e^{-nr})\right) e^{tn \epsilon}.
  \end{align*}
\end{theorem}

At first, we provide the proof of Theorem \ref{BPS-theorem} for the ergodic case. 
Then we explain how to modify the proof for the nonergodic case. 
We adapt the method of \cite{BPS-99-annals} to make the properties of partitions suitable 
for the $r$-neutralized Bowen ball. 

{\bf The case of ergodic measures: }

Let $\mu$ be a hyperbolic ergodic $f$-invariant Borel probability measure on $M$. 

Given $ 0<\epsilon<1 $, there are a set $ \Gamma \subset M $ with $ \mu(\Gamma)>1-\frac{\epsilon}{4} $, an integer $n_{0} \geq 1 $, 
and a constant $  C>1 $ such that for every $ x \in \Gamma $ and any integer $ n \geq n_{0} $, the properties (a),(c),(d),(e) and the following statements hold:
\begin{itemize}
\item[(f)]Based on the properties \eqref{P-(8)} and \eqref{P-(9)}, we have
\begin{align*}
\mathcal{P}_{a nr}^{a n(r+\chi_1)}(x) &\subset B\left(x, n, e^{-nr}\right) \subset \mathcal{P}(x), \\ 
\mathcal{P}_{a nr}^{0}(x) \cap \xi^{s}(x) &\subset B^{s}\left(x, e^{-nr}\right) \subset \mathcal{P}(x) \cap \xi^{s}(x), \\
\mathcal{P}_{0}^{a n(r+\chi_1)}(x) \cap \xi^{u}(x) &\subset B^{u}\left(x,n, e^{-nr}\right) \subset \mathcal{P}(x) \cap \xi^{u}(x),  
\end{align*}
where $ a$ is the integer part of $ 2(1+\epsilon) \max \left\{\lambda_{1},-\lambda_{m}, 1\right\} $ and $\chi_1$ is the integer part of $\max\{\lambda_1+1,-\lambda_m+1\}$. 
\item[(g)] We define
$Q_{n}(x):=\bigcup \mathcal{P}_{a nr}^{a n(r+\chi)}(y)$, 
where the union is taken over $ y \in \Gamma $ for which
\[
\mathcal{P}_{0}^{a n (r+\chi_1)}(y) \cap B^{u}\left(x, n, 2 e^{-nr}\right) \neq \varnothing \text { and } \mathcal{P}_{a nr}^{0}(y) \cap B^{s}\left(x,2 e^{-nr}\right) \neq \varnothing .
\]
Then by the continuous dependence of stable and unstable manifolds in the $C^{1+\alpha}$ topology on the base point, 
we obtain  
\begin{align}\label{P-(16)}
B\left(x, n,e^{-nr}\right) \cap \Gamma \subset Q_{n}(x) \subset B\left(x,n,4 e^{-nr}\right) ,
\end{align}
and for each $ y \in Q_{n}(x) $, we have 
\[
\mathcal{P}_{anr}^{an(r+\chi_1)}(y) \subset Q_{n}(x).
\]
\item[(h)] For every $ x \in \Gamma $ and $ n \geq n_{0} $ (increasing $ n_{0} $ if necessary), we have
\begin{align}
B^{s}\left(x,e^{-nr}\right) \cap \Gamma \subset Q_{n}(x) \cap \xi^{s}(x) \subset B^{s}\left(x,4 e^{-nr}\right), \label{P-(17)}\\
B^{u}\left(x,n, e^{-nr}\right) \cap \Gamma \subset Q_{n}(x) \cap \xi^{u}(x) \subset B^{u}\left(x,n,4 e^{-nr}\right). \label{P-(18)}
\end{align}
\end{itemize}
Since $\mu$ is ergodic, the number $ h $ of property (a) is the  measure-theoretic entropy $h_\mu(f)$. 

We now introduce some propositions and notations from \cite{BPS-99-annals}. 

\begin{proposition}[\cite{BPS-99-annals}]\label{Proposition 4}
  There exists a positive constant $ D=D(\Gamma_0)<1 $ such that for every $ k \geq 1 $ and $ x \in \Gamma $, we have 
    \begin{align*}
      \mu_{x}^{s}\left(\mathcal{P}_{0}^{k}(x) \cap \Gamma\right) \geq D, \quad
      \mu_{x}^{u}\left(\mathcal{P}_{k}^{0}(x) \cap \Gamma\right) \geq D.
    \end{align*}
\end{proposition}

\begin{proposition}[\cite{BPS-99-annals}]\label{Proposition 5}
  For every $ x \in \Gamma $ and $ n \geq n_{0}r^{-1} $, we have
  \begin{align*}
    \mathcal{P}_{anr}^{an(r+\chi_1)}(x) \cap \xi^{s}(x)=\mathcal{P}_{anr}^{0}(x) \cap \xi^{s}(x), \\
    \mathcal{P}_{anr}^{an(r+\chi_1)}(x) \cap \xi^{u}(x)=\mathcal{P}_{0}^{an(r+\chi_1)}(x) \cap \xi^{u}(x).
  \end{align*}
\end{proposition}
 
Fix $ x \in \Gamma $ and an integer $ n \geq n_{0} $. 
  We introduce two classes $\mathcal{R}(n)$ and $\mathcal{F}(n)$ as follows. 
\begin{align*}
\mathcal{R}(n)&:=\left\{\mathcal{P}_{anr}^{an(r+\chi_1)}(y) \subset \mathcal{P}(x): \mathcal{P}_{anr}^{an(r+\chi_1)}(y) \cap \Gamma \neq \varnothing\right\}; \\
\mathcal{F}(n)&:=\left\{\mathcal{P}_{anr}^{an(r+\chi_1)}(y) \subset \mathcal{P}(x): \mathcal{P}_{anr}^{0}(y) \cap \Gamma \neq \varnothing \text { and } \mathcal{P}_{0}^{an(r+\chi_1)}(y) \cap \Gamma \neq \varnothing\right\}. 
\end{align*}
  We call the elements of these classes  ``rectangles''. 
  
For each set $ A \subset \mathcal{P}(x) $, we have the following definitions. 
\begin{align*}
N(n, A)&:=\operatorname{Card}\{R \in \mathcal{R}(n): R \cap A \neq \varnothing\}; \\
N^{s}(n, y, A)&:=\operatorname{Card}\left\{R \in \mathcal{R}(n): R \cap \xi^{s}(y) \cap \Gamma \cap A \neq \varnothing\right\}; \\
N^{u}(n, y, A)&:=\operatorname{Card}\left\{R \in \mathcal{R}(n): R \cap \xi^{u}(y) \cap \Gamma \cap A \neq \varnothing\right\}; \\
\widehat{N}^{s}(n, y, A)&:=\operatorname{Card}\left\{R \in \mathcal{F}(n): R \cap \xi^{s}(y) \cap A \neq \varnothing\right\}; \\
\widehat{N}^{u}(n, y, A)&:=\operatorname{Card}\left\{R \in \mathcal{F}(n): R \cap \xi^{u}(y) \cap A \neq \varnothing\right\}. 
\end{align*}

The next lemma is a corollary of Proposition \ref{Proposition 5}.

\begin{lemma}\label{Lemma 1}
For each $ y \in \mathcal{P}(x) \cap \Gamma $ and integer $ n \geq n_{0} $, we have:
\begin{align*}
N^{s}\left(n, y, Q_{n}(y)\right) \leq C\mu_{y}^{s}\left(B^{s}\left(y, 4 e^{-nr}\right)\right) \exp(anr (h+ \epsilon)); \\
N^{u}\left(n, y, Q_{n}(y)\right) \leq C \mu_{y}^{u}\left(B^{u}\left(y, n, 4 e^{-nr}\right)\right)\exp(a(r+\chi_1) n (h+ \epsilon)). 
\end{align*}
\end{lemma}

\begin{proof}
We will prove the second inequality, the proof of the first inequality is similar.

Let $ z \in   R \cap \xi^{s}(y) \cap Q_{n}(y) \cap \Gamma $ for some $ R \in \mathcal{R}(n) $. 
By applying Proposition \ref{Proposition 5} and \eqref{LY-(c)}, we have 
\[
  \mu_{y}^{u}(R)=\mu_{y}^{u}\left(\mathcal{P}_{0}^{a n(r+\chi_1)}(z)\right)=\mu_{z}^{u}\left(\mathcal{P}^{a n(r+\chi_1)}_{0}(z)\right)\geq C^{-1}e^{-a(r+\chi_1)n(h+\epsilon)}.
\]
Since \eqref{P-(18)} and $ R \cap \xi^{u}(y) \cap Q_{n}(y) \neq \varnothing $ implies $ R \in \mathcal{R}(n) $, we obtain
\begin{align*}
&\mu_{y}^{u}\left(B^{u}\left(y,n, 4 e^{-nr}\right)\right) \geq \mu_{y}^{u}\left(Q_{n}(y)\right) \\
\geq& N^{u}\left(n, y, Q_{n}(y)\right) \cdot \min \left\{\mu_{y}^{u}(R): R \in \mathcal{R}(n) \text { and } R \cap \xi^{u}(y) \cap Q_{n}(y) \cap \Gamma \neq \varnothing\right\}\\
\geq & N^{u}\left(n, y, Q_{n}(y)\right) \cdot C^{-1}\exp(-a(r+\chi_1)n(h+\epsilon)). 
\end{align*}
 Thus the proof of the second inequality follows. 
\end{proof}

\begin{lemma}\label{Lemma 2}
For each $ y \in \mathcal{P}(x) \cap \Gamma $ and integer $ n \geq n_{0} $, we have
\[
\mu\left(B\left(y, n,e^{-nr}\right)\right) \leq C  N\left(n, Q_{n}(y)\right) \exp(n\epsilon- a n(2r+\chi_1) (h- \epsilon)). 
\]
\end{lemma}

\begin{proof}
  Since $ R \cap Q_{n}(y) \neq \varnothing $ implies $ R \in \mathcal{R}(n)$, 
    by \eqref{LY-(a)}, \eqref{P-(19)} and \eqref{P-(16)}, we derive
\begin{align*}
\mu\left(B\left(y, n,e^{-nr}\right)\right)&\leq e^{n\epsilon} \mu\left(B\left(y, n, e^{-nr}\right) \cap \Gamma\right) \leq e^{n\epsilon} \mu\left(Q_{n}(y) \cap \Gamma\right) \\
&\leq e^{n\epsilon} N\left(n, Q_{n}(y)\right) \cdot \max \left\{\mu(R): R \in \mathcal{R}(n) \text { and } R \cap Q_{n}(y) \neq \varnothing\right\}\\
&\leq  N\left(n, Q_{n}(y)\right) \cdot \exp (n\epsilon- a n(2r+\chi_1) (h- \epsilon)).
\end{align*}
\end{proof}

Let $b$ be the integer part of $r^{-1}\log 4 +1$. 
Next, we estimate the number of elements in the classes $ \mathcal{R}(n) $ and $ \mathcal{F}(n) $. 

\begin{lemma}\label{Lemma 3}
For $ \mu $-a.e. $ y \in \mathcal{P}(x) \cap \Gamma $ and $n\geq n_0$, 
there exist a constant $C_1$ such that 
\[
N\left(n+b, Q_{n+b}(y)\right) \leq C_1 \widehat{N}^{s}\left(n, y, Q_{n}(y)\right) \cdot \widehat{N}^{u}\left(n, y, Q_{n}(y)\right) \exp(n\epsilon+2an(2r+\chi_1)\epsilon). 
\]
\end{lemma}

\begin{proof}
It follows from \eqref{P-(19)} and \eqref{P-(16)} that for any $ n \geq n_0 $,
\begin{align}
e^{n\epsilon} \mu\left(Q_{n}(y) \cap \Gamma\right) & \geq e^{n\epsilon}  \mu\left(B\left(y, n,e^{-nr}\right) \cap \Gamma\right) \geq \mu\left(B\left(y, n,e^{-nr}\right)\right)\nonumber \\
& \geq \mu\left(B\left(y, n, 4 e^{-r(n+b)}\right)\right) \geq \mu\left(Q_{n+b}(y)\right). \label{P-(22)}
\end{align}
According to \eqref{LY-(a)}, we derive
\begin{align}
\mu\left(Q_{m}(y)\right)&=\sum_{\mathcal{P}_{a mr}^{a m(r+\chi_1)}(z) \subset Q_{m}(y)} \mu\left(\mathcal{P}_{a mr}^{a m(r+\chi_1)}(z)\right) \nonumber\\
&\geq N\left(m, Q_{m}(y)\right)  C^{-1} \exp(- a m(2r+\chi_1) (h+\epsilon));\label{L-1}\\
\mu\left(Q_{m}(y) \cap \Gamma\right)&=\sum_{\substack{\mathcal{P}_{amr}^{a m(r+\chi_1)}(z) \subset Q_{n}(y)}} \mu\left(\mathcal{P}_{amr}^{am(r+\chi_1)}(z) \cap \Gamma\right) \nonumber\\
&\leq N_{m} C \exp(-am (2r+\chi_1) (h-\epsilon)),\label{L-2}
\end{align}
where $ N_{m} $ is the number of rectangles $ \mathcal{P}_{amr}^{am(r+\chi_1)}(z) \in \mathcal{R}(m)$  with $\mathcal{P}_{amr}^{am(r+\chi_1)}(z)\cap \Gamma \neq\varnothing$.
Setting $ m=n+b$ and using \eqref{P-(22)}, \eqref{L-1}, \eqref{L-2}, we obtain
\begin{align}
  N\left(n+b, Q_{n+b}(y)\right) &\leq \mu\left(Q_{n+b}(y)\right) \cdot C\exp(a (n+b)(2r+\chi_1) (h+\epsilon))\nonumber \\
&\leq  \mu\left(Q_{n}(y) \cap \Gamma\right) \cdot C\exp(n\epsilon+a (n+b)(2r+\chi_1) (h+\epsilon))\nonumber \\
  &\leq N_{n} \cdot   C^{2}\exp(n\epsilon+2an(2r+\chi_1)\epsilon+ab(2r+\chi_1) (h+\epsilon))\nonumber\\
&\leq N_{n} \cdot   C_1\exp(n\epsilon+2an(2r+\chi_1)\epsilon), \label{P-(23)}
\end{align}
where we choose $C_1=C^2\exp(ab(2r+\chi_1) (h+\epsilon))$. For any $ y \in \Gamma$, we have
\begin{align*}
\mathcal{P}_{0}^{a n(r+\chi_1)}(y) \cap \xi^{u}(y) \cap \Gamma \neq \varnothing \text{ and } \mathcal{P}_{a nr}^{0}(y) \cap \xi^{s}(y) \cap \Gamma \neq \varnothing.
\end{align*}

Consider a rectangle $ \mathcal{P}_{a nr}^{an(r+\chi_1)}(v) \subset Q_{n}(y)$ with $\mathcal{P}_{a nr}^{an(r+\chi_1)}(v)\cap \Gamma \neq \varnothing$. 
Therefore, the rectangles $ \mathcal{P}_{anr}^{0}(v) \cap \mathcal{P}_{0}^{an(r+\chi_1)}(y) $ and 
$ \mathcal{P}_{anr}^{0}(y) \cap \mathcal{P}_{0}^{an(r+\chi_1)}(v) $ are in $ \mathcal{F}(n) $ and 
they intersect the stable and unstable local manifolds at $ y $ respectively. 
Then, for any rectangle $ \mathcal{P}_{anr}^{an(r+\chi_1)}(v) \subset Q_{n}(y) $ with $ \mathcal{P}_{anr}^{an(r+\chi_1)}(v) \cap \Gamma \neq \varnothing$, 
we can associate it with the pair of rectangles $ \left(\mathcal{P}_{anr}^{0}(v) \cap \mathcal{P}_{0}^{an(r+\chi_1)}(y), \mathcal{P}_{anr}^{0}(y) \cap \mathcal{P}_{0}^{an(r+\chi_1)}(v)\right) $ in
\[
\left\{R \in \mathcal{F}(n): R \cap \xi^{s}(y) \cap Q_{n}(y) \neq \varnothing\right\} \times\left\{R \in \mathcal{F}(n): R \cap \xi^{u}(y) \cap Q_{n}(y) \neq \varnothing\right\}. 
\]
Obviously, this correspondence is injective, 
 as each rectangle in $Q_n(y)$ intersecting $\Gamma$
 will correspond to a unique pair of rectangles in $\mathcal{F}(n)$. 
Thus we conclude 
\begin{align}\label{L-3}
\widehat{N}^{s}\left(n, y, Q_{n}(y)\right) \cdot \widehat{N}^{u}\left(n, y, Q_{n}(y)\right) \geq N_{n}. 
\end{align}
Combining \eqref{P-(23)} and \eqref{L-3}, we obtain the desired inequality, 
thus completing the proof of the lemma.
\end{proof}

In the following lemma, we give upper bounds for the number of rectangles in $\mathcal{F}(n)$.

\begin{lemma}\label{Lemma 4}
For each $ x \in \Gamma $, there exists a constant $C_2$ such that for any $n\geq n_0$, 
\begin{align*}
   \widehat{N}^{s}(n, x, \mathcal{P}(x)) \leq C_2 \exp(a n (rh+3r\epsilon+2\chi_1\epsilon));\\
 \widehat{N}^{u}(n, x, \mathcal{P}(x)) \leq  C_2 \exp(a n (rh+3r\epsilon+\chi_1\epsilon+\chi_1 h)). 
\end{align*}
\end{lemma}

\begin{proof}
We will prove the first inequality, the proof of the second inequality is similar. 

 Since the partition $ \mathcal{P} $ is countable, 
 there exists a sequence of points $\{ y_{i} \}_i$ such that 
$\bigcup_i \mathcal{P}^{0}_{anr}\left(y_{i}\right)= \mathcal{P}(x) $, 
and these rectangles are mutually disjoint. 
Without loss of generality, we assume that 
for any $i$ where $ \mathcal{P}^{0}_{anr}\left(y_{i}\right) \cap \Gamma \neq \varnothing $, 
we have $ y_{i} \in \Gamma $. 
  Let $\tilde{I}$ be the set of indices $i$ for which 
  $\mathcal{P}^{0}_{anr}\left(y_{i}\right) \cap \Gamma \neq \varnothing$, then
\begin{align}
N(n, \mathcal{P}(x)) &\geq \sum_{i} N^{u}\left(n, y_{i}, \mathcal{P}^{0}_{anr}\left(y_{i}\right)\right)\geq \sum_{i\in \tilde{I}} N^{u}\left(n, y_{i}, \mathcal{P}^{0}_{anr}\left(y_{i}\right)\right). \label{P-(24)} 
\end{align}
For any $y_i\in \Gamma$, by Proposition \ref{Proposition 4}, Proposition \ref{Proposition 5} 
and \eqref{LY-(b)}, we conclude
\begin{align}
N^{u}\left(n, y_{i}, \mathcal{P}^{0}_{anr}\left(y_{i}\right)\right) & \geq \frac{\mu_{y_{i}}^{u}\left(\mathcal{P}^{0}_{anr}\left(y_{i}\right) \cap \Gamma\right)}{\max \left\{\mu_{z}^{u}\left(\mathcal{P}_{anr}^{an(r+\chi_1)}(z)\right): z \in \xi^{u}\left(y_{i}\right) \cap \mathcal{P}(x) \cap \Gamma\right\}} \nonumber\\
& \geq \frac{D}{\max \left\{\mu_{z}^{u}\left(\mathcal{P}_{anr}^{an(r+\chi_1)}(z)\right): z \in \xi^{u}\left(y_{i}\right) \cap \mathcal{P}(x) \cap \Gamma\right\}} \nonumber\\
& =\frac{D}{\max \left\{\mu_{z}^{u}\left(\mathcal{P}_{0}^{an(r+\chi_1)}(z)\right): z \in \xi^{u}\left(y_{i}\right) \cap \mathcal{P}(x) \cap \Gamma\right\}} \nonumber\\
& \geq D C^{-1} e^{a n(r+\chi_1)( h- \epsilon)}.  \label{P-(25)}
\end{align}
Through \eqref{LY-(a)}, we obtain
\begin{align}\label{P-(26)}
N(n, \mathcal{P}(x)) \leq \frac{\mu(\mathcal{P}(x))}{\min \left\{\mu\left(\mathcal{P}_{anr}^{an(r+\chi_1)}(z)\right): z \in \mathcal{P}(x) \cap \Gamma\right\}} \leq  C e^{ a n(2r+\chi_1) (h+\epsilon)}. 
\end{align}
By the definition of $\hat{N}^{s}(n, x, \mathcal{P}(x))$, we have
\begin{align*}
\widehat{N}^{s}(n, x, \mathcal{P}(x))=\operatorname{Card}\left\{i: \mathcal{P}^{0}_{anr}\left(y_{i}\right) \cap \Gamma \neq \varnothing\right\}. 
\end{align*}
Combining \eqref{P-(24)}, \eqref{P-(25)} and \eqref{P-(26)}, we conclude
$$
\begin{aligned}
  C e^{ a n(2r+\chi_1) (h+\epsilon)} & \geq N(n, \mathcal{P}(x))  \geq \sum_{i: \mathcal{P}^{0}_{anr}\left(y_{i}\right) \cap \hat{\Gamma} \neq \varnothing} N^{u}\left(n, y_{i}, \mathcal{P}^{0}_{anr}\left(y_{i}\right)\right) \\
& \geq \hat{N}^{s}(n, x, \mathcal{P}(x)) \cdot D C^{-1} \exp(a n(r+\chi_1)(h-\epsilon)). 
\end{aligned}
$$
Thus the first inequality follows by choosing $C_2=D^{-1}C^2$. 
\end{proof}

We continue with a covering lemma proved in \cite{O24}. 

\begin{lemma}[\cite{O24}]\label{puck covers}
  Let $f$ be a $ C^{1+\alpha} $ diffeomorphism on a closed Riemannian manifold $M$ with $m=\operatorname{dim}M\geq 2$ and  $\mu $ an $ f$-invariant Borel probability measure on $ M $.
  There exists a constant $C_m$ which depends only on $m$, such that
  for $\mu$-a.e. $x\in M$ and any $n\in \mathbb{N}$, every set $A\subset \xi^u(x)$ 
  can be covered by $r$-neutralized Bowen balls $\{B(x_i,n, e^{-nr})\}$ 
  with the property that no point in $A$ is contained in more than $C_m$ of these balls.
\end{lemma}

We are now ready to prove the following lemma. 
It provides a comparison between the number of rectangles
 in $ \mathcal{F}(n) $ and the number of rectangles in $ \mathcal{R}(n) $.

\begin{lemma}\label{Lemma 5}
For $ \mu $-almost every $ y \in \mathcal{P}(x) \cap \Gamma $, we have
\begin{align*}
\varlimsup\limits_{n \rightarrow+\infty} \dfrac{\widehat{N}^{s}\left(n, y, Q_{n}(y)\right)}{N^{s}\left(n, y, Q_{n}(y)\right)} \exp(-7 a nr \epsilon-2an\chi_1\epsilon)<1; \\
\varlimsup\limits_{n \rightarrow+\infty} \dfrac{\hat{N}^{u}\left(n, y, Q_{n}(y)\right)}{N^{u}\left(n, y, Q_{n}(y)\right)} \exp(-7 a nr \epsilon-2an\chi_1\epsilon)<1. 
\end{align*}
\end{lemma}

\begin{proof}
We will only prove the second inequality, as the
 proof of the first inequality is almost identical to the proof of \cite[Lemma 5]{BPS-99-annals}.

At first, let us define the set $F$ in the following, 
\[
  F:=\left\{y \in \Gamma: \varlimsup_{n \rightarrow+\infty} \frac{\hat{N}^{u}\left(n, y, Q_{n}(y)\right)}{N^{u}\left(n, y, Q_{n}(y)\right)} \exp(-7 a nr \epsilon-2an\chi_1\epsilon) \geq 1\right\}.
\]
To prove the lemma, it suffices to show that $\mu(F)=0$. 
We proceed by contradiction, assuming that $ \mu(F)>0 $. 
Combining \eqref{LY-(f)}and \eqref{P-(18)}, for each $ n \geq n_{0} $ and $ y \in \Gamma $, we have
\begin{align}\label{L-4}
\mu_{y}^{u}\left(Q_{n}(y)\right) \geq \mu_{y}^{u}\left(B^{u}\left(y,n, e^{-nr}\right) \cap \Gamma\right) \geq e^{-n\epsilon}\mu_{y}^{u}\left(B^{u}\left(y,n, e^{-nr}\right)\right). 
\end{align}
By applying \eqref{LY-(c)}, \eqref{P-(11)}, \eqref{L-4} and Proposition \ref{Proposition 5}, we conclude
\begin{align}
N^{u}\left(n, y, Q_{n}(y)\right) & \geq \frac{\mu_{y}^{u}\left(Q_{n}(y)\right)}{\max \left\{\mu_{z}^{u}\left(\mathcal{P}_{anr}^{an(r+\chi_1)}(z)\right): z \in \xi^{u}(y) \cap \mathcal{P}(x) \cap \Gamma\right\}}\nonumber \\
& \geq e^{-n\epsilon} \frac{\mu_{y}^{u}\left(B^{u}\left(y, n, e^{-nr}\right)\right)}{\max \left\{\mu_{z}^{u}\left(\mathcal{P}_{0}^{an(r+\chi_1)}(z)\right): z \in \xi^{u}(y) \cap \mathcal{P}(x) \cap \Gamma\right\}} \nonumber \\
& \geq C^{-1}\exp(-n\epsilon-d^{u} nr-hn-nr \epsilon+a n(r+\chi_1) (h+\epsilon)). \label{P-(28)}
\end{align}
For each $ y \in F $, there exists an increasing sequence $ \left\{m_{j}\right\}_{j=1}^{\infty}=\left\{m_{j}(y)\right\}_{j=1}^{\infty} $ of positive integers such that for any $j$, 
\begin{align}
\widehat{N}^{u}\left(m_j, y, Q_{m_j}(y)\right)\geq \frac{1}{2} N^{u}\left(m_{j}, y, Q_{m_j}(y)\right) \exp(7 a m_jr \epsilon+2am_j\chi_1\epsilon) .\label{P-(29)}
\end{align}
Let $ F^{\prime} \subset F $ be the set of points $ y \in F $ for which the following limit exists 
\[
\lim _{n \rightarrow \infty} \frac{\log \mu_{y}^{u}\left(B^{u}(y,n, e^{-nr})\right)}{n}=rd^{u}+h. 
\]
Obviously, we have $ \mu\left(F^{\prime}\right)=\mu(F)>0 $. 
Then we can find $ y \in F $ with 
\begin{align*}
\mu_{y}^{s}(F)=\mu_{y}^{s}\left(F^{\prime}\right)=\mu_{y}^{s}\left(F^{\prime} \cap \mathcal{P}(y) \cap \xi^{s}(y)\right)>0,\\
h^r_{\mu^u_y,d}(f,x)=rd^u+h,\quad \text{ for any } x\in F^{\prime} \cap \xi^{u}(y). 
\end{align*}

To finish the proof of this lemma, we need to use the dimension theory of $r$-neutralized local entropy. 
We refer the reader to the Appendix for the definition and properties of the $r$-neutralized Bowen topological entropy $h^{B}_{r,d}\left(f,\cdot\right)$. 
By Lemma \ref{Billingsley type}, we obtain
\begin{align}\label{P-(30)}
h^{B}_{r,d}\left(f,F^{\prime} \cap \xi^{u}(y)\right)\geq h+rd^{u}. 
\end{align}
Let us consider the countable collection of balls
\[
  \mathfrak{B}=\left\{B\left(z,m_{j}(z), 4 e^{-m_{j}(z)r}\right): z \in F^{\prime} \cap \xi^{u}(y),\quad j=1,2, \ldots\right\}.
\]
Applying Lemma \ref{puck covers}, for any $ L>0 $, there exists a sequence of points $ \left\{z_{i} \in F^{\prime} \cap \xi^{s}(y)\right\}_{i=1}^{\infty} $ and
a sequence of integers $ \left\{t_{i}\right\}_{i=1}^{\infty} $, where $ t_{i} \in\left\{m_{j}\left(z_{i}\right)\right\}_{j=1}^{\infty} $ and
$t_{i}>L $ for each $ i $ such that we can find a subcover $\mathfrak{C}\subset \mathfrak{B}$ of $ F^{\prime} \cap \xi^{s}(y) $,
\[
  \mathfrak{C}=\left\{B\left(z_{i}, t_i, 4 e^{-t_{i}r}\right): i=1,2, \ldots\right\},
\]
and each set $ Q_{t_i}(z_i) $ appears in the sum $ \sum\limits_{i: t_{i}=q} \widehat{N}^{s}\left(t_i, z_{i}, Q_{t_i}(z_i)\right) $ at most $C_m$ times, then
\begin{align}\label{L-5}
  \sum_{i: t_{i}=q} \hat{N}^{s}\left(q, z_{i}, Q_{t_i}(z_i)\right) \leq C_m \hat{N}^{s}(q, y, \mathcal{P}(y)).
\end{align}
Based on \eqref{P-(29)}, \eqref{L-5} and Lemma \ref{Lemma 4}, we obtain
\begin{align*}
 &M^{h+rd^u-\epsilon}_{L,r,d}\left(f,F^{\prime} \cap \xi^{u}(y)\right)\leq  \sum_{i=1}^{\infty} \exp(-t_{i}\left(h+rd^{u}-\epsilon\right)r) \\
\leq &\sum_{i=1}^{\infty} \hat{N}^{u}\left(t_{i}, z_{i}, Q_{t_i}(z_i)\right) \cdot 2 C \exp(t_i\epsilon+2t_ir\epsilon-a t_{i}(r+\chi_1) h-7 a t_{i}r \epsilon-2at_{i}\chi_1\epsilon)   \\
\leq & 2 C \sum_{q=1}^{\infty} \exp(q\epsilon+2qr\epsilon-a q(r+\chi_1) h-7 a qr \epsilon-2aq\chi_1\epsilon) \sum_{i: t_{i}=q} \widehat{N}^{u}\left(q, z_{i}, Q_{t_i}(z_i)\right) \\
\leq& 2 C \sum_{q=1}^{\infty} \exp(q\epsilon+2qr\epsilon-a q(r+\chi_1) h-7 a qr \epsilon-2aq\chi_1\epsilon) \cdot C_m \widehat{N}^{u}\left(q, y, \mathcal{P}(y)\right) \\
\leq& 2 C C_m C_2 \sum_{q=1}^{\infty} \exp(q\epsilon+2qr\epsilon-4a qr\epsilon- a q\chi_1\epsilon) <\infty. 
\end{align*}
Since $ L $ and $ t_{i} $ can be chosen sufficiently large, we conclude that
\[
h^{B}_{r,d}\left(f,F^{\prime} \cap \xi^{u}(y)\right) \leq h+rd^{u}-\epsilon<h+rd^{u},
\]
which contradicts \eqref{P-(30)}. 
Therefore, we have $ \mu(F)=0 $, proving the second inequality.
\end{proof}

By Lemma \ref{Lemma 5}, for $ \mu $-a.e. $y \in \mathcal{P}(x) \cap \Gamma $, 
there exists an integer $ n_{1}(y) \geq n_{0} $ such that 
for all $ n \geq n_{1}(y) $, we have
\begin{align}
\widehat{N}^{s}\left(n, y, Q_{n}(y)\right)<N^{s}\left(n, y, Q_{n}(y)\right) \exp(7 a nr \epsilon+2an\chi_1\epsilon); \label{P-(31)}\\
\widehat{N}^{u}\left(n, y, Q_{n}(y)\right)<N^{u}\left(n, y, Q_{n}(y)\right) \exp(7 a nr \epsilon+2an\chi_1\epsilon). \label{P-(32)}
\end{align}
Applying Lusin's theorem, for any $ \epsilon>0 $, there exists a subset $ \Gamma_{\epsilon} \subset \Gamma $ satisfying
\[
\mu\left(\Gamma_{\epsilon}\right)>\mu(\Gamma)-\frac{\epsilon}{2}>1-\epsilon, \quad \text { and }\quad n_{\epsilon} \stackrel{\text { def }}{=} \sup \left\{n_{0}, n_{1}(y): y \in \Gamma_{\epsilon}\right\}<\infty,
\]
such that the inequalities \eqref{P-(31)} and \eqref{P-(32)} hold for every $ n \geq n_{\epsilon} $.

We are ready to prove Theorem \ref{BPS-theorem} in the following.

\begin{lemma}\label{Lemma 6}
 For every $ \epsilon>0 $, if $ y \in \Gamma_{\epsilon} $ and $ n \geq n_{\epsilon} $, then
\begin{align*}
\mu_{y}^{s}\left(B^{s}\left(y, e^{-nr}\right)\right) \mu_{y}^{u}\left(B^{u}\left(y,n, e^{-nr}\right)\right) \leq \mu\left(B\left(y, n, 4 e^{-nr}\right)\right) \cdot 4 C^{3} \exp( a n (11r+3\chi_1) \epsilon). 
\end{align*}
\end{lemma}

\begin{proof}
For any $ z \in \Gamma_{\epsilon} \cap Q_{n}(y) $ and $ n \geq n_{\epsilon} $, by \eqref{P-(32)}, we have
\begin{align}
N^{u}\left(n, y, Q_{n}(y)\right) &\leq \widehat{N}^{u}\left(n, y, Q_{n}(y)\right)=\widehat{N}^{u}\left(n, z, Q_{n}(y)\right)\nonumber\\
&<N^{u}\left(n, z, Q_{n}(y)\right) \exp(7 a n \epsilon+2an\chi_1\epsilon),\nonumber\\
N^{u}\left(n, y, Q_{n}(y)\right) &\leq \inf \left\{N^{u}\left(n, z, Q_{n}(y)\right): z \in \Gamma_{\epsilon} \cap Q_{n}(y)\right\} \exp(7 a n \epsilon+2an\chi_1\epsilon). \label{P-(33)}
\end{align}
Since $ N\left(n, Q_{n}(y)\right) $ is equal to the number of rectangles $ R \subset Q_{n}(y) $, we obtain 
\[
\widehat{N}^{s}\left(n, y, Q_{n}(y)\right) \times \inf \left\{N^{u}\left(n, z, Q_{n}(y)\right): z \in Q_{n}(y)\right\} \leq N\left(n, Q_{n}(y)\right). 
\]
For $ y \in \Gamma_{\epsilon} $ and $ n \geq n_{\epsilon} $, by using \eqref{P-(28)} and \eqref{P-(33)}, we derive
\begin{align}
&N^{s}\left(n, y, Q_{n}(y)\right) \times N^{u}\left(n, y, Q_{n}(y)\right) \leq N\left(n, Q_{n}(y)\right) \exp(7 a nr \epsilon+an\chi_1\epsilon);\label{P-1}\\
&N^{s}\left(n, y, Q_{n}(y)\right) \geq \mu_{y}^{s}\left(B^{s}\left(y, e^{-nr}\right)\right) \cdot(2 C)^{-1} \exp(a nr( h- \epsilon)); \label{P-2}\\
&N^{u}\left(n, y, Q_{n}(y)\right) \geq \mu_{y}^{u}\left(B^{u}\left(y, e^{-(r+\chi_1)n}\right)\right) \cdot(2 C)^{-1} \exp(a n(r+\chi_1) (h- \epsilon)). \label{P-3}
\end{align}
Through \eqref{LY-(a)} and \eqref{P-(16)}, we obtain
\begin{align}
N\left(n, Q_{n}(y)\right) &\leq \frac{\mu\left(Q_{n}(y)\right)}{\min \left\{\mu\left(\mathcal{P}_{anr}^{an(r+\chi_1)}(z)\right): z \in Q_{n}(y) \cap \Gamma\right\}} \nonumber\\
&\leq \mu\left(B(y,n, 4 e^{-nr})\right) \cdot C \exp( a n(2r+\chi_1) (h+\epsilon)). \label{P-4}
\end{align}
Combining \eqref{P-1}, \eqref{P-2}, \eqref{P-3} and \eqref{P-4}, we obtain the desired statement.
\end{proof}

The following lemma provides a upper bound for the measure of the $r$-neutralized Bowen ball $B(y, n+b,e^{-(n+b)r})$. 
Recall that $b$ is the integer part of $r^{-1}\log 4 +1$. 

\begin{lemma}\label{Lemma 7}
  For $ \mu $-a.e. $ y \in \mathcal{P}(x) \cap \Gamma_\epsilon$ and any $n\geq n_\epsilon$, we have
\begin{align*}
       & \mu\left(B(y, n+b,e^{-(n+b)r})\right)                                                                                                                \\
& \leq C^3C_1\mu_{y}^{s}\left(B^{s}(y,n, 4 e^{-nr})\right) \mu_{y}^{u}\left(B^{u}(y, n,4 e^{-nr})\right)  \exp(16n\epsilon+an(2r+\chi_1)\epsilon). 
\end{align*}
\end{lemma}

\begin{proof}
By Lemma \ref{Lemma 2} and Lemma \ref{Lemma 3},
for $\mu$-a.e. $ y \in \mathcal{P}(x) \cap \Gamma_\epsilon $ and $ n \geq n_\epsilon $, we have
\begin{align*}
       & \mu\left(B(y, n+b, e^{-(n+b)r})\right)\nonumber                                                                                                      \\
  \leq & C N\left(n+b, Q_{n+b}(y)\right)  \exp((n+b)\epsilon- a (n+b)(2r+\chi_1) (h- \epsilon))                                                               \\
  \leq & CC_1\widehat{N}^{s}\left(n, y, Q_{n}(y)\right)  \widehat{N}^{u}\left(n, y, Q_{n}(y)\right) \exp(-a n(2r+\chi_1) h+n\epsilon+2an(2r+\chi_1)\epsilon).
\end{align*}
Finally, through \eqref{P-(31)}, \eqref{P-(32)} and Lemma \ref{Lemma 1}, we obtain
\begin{align*}
       & \mu\left(B(y, n+b,e^{-(n+b)r})\right)                                                                                                       \\
  \leq & CC_1N^{s}\left(n, y, Q_{n}(y)\right)  N^{u}\left(n, y, Q_{n}(y)\right)\exp(-an(2r+\chi_1)h+15an(2r+\chi_1)\epsilon)                         \\
  \leq & C^3C_1\mu_{y}^{s}\left(B^{s}(y, 4 e^{-nr})\right) \mu_{y}^{u}\left(B^{u}(y, n,4 e^{-nr})\right)  \exp(16n\epsilon+an(2r+\chi_1)\epsilon).
\end{align*}
This completes the proof of the lemma.
\end{proof}

We prove Theorem \ref{BPS-theorem} by combining Lemma \ref{Lemma 6} and Lemma \ref{Lemma 7},
and choosing an appropriate constant $t$. Then
\eqref{P-(10)}, \eqref{P-(11)}, and Theorem \ref{BPS-theorem}
imply Theorem \ref{BPS-theorem-1}. 

{\bf The case of nonergodic measures: }

We will demonstrate how to adapt the argument above to provide the proof of 
Theorem \ref{BPS-theorem} 
in the case that the invariant measure $ \mu $ is not ergodic. 

For $ \mu $-almost every $ x \in M $, $ d^{s}(x) $ and $ d^{u}(x) $ exist but may vary with $x$. 
Given $ \delta>0$, there exists a measurable subset $ \Gamma^{*}\subset M $ with 
$ \mu(\Gamma^{*}) \geq 1-\delta $, such that 
for every point $ x \in \Gamma^{*} $, 
the properties (a)-(e) (with the exception of (c)) hold if $ h,d^s,d^u $ are replaced by $ h_\mu(f, x),  d^{s}(x),  d^{u}(x) $ respectively. 

Let us fix $ \epsilon>0 $ and consider the sets
\[
\Gamma(x)=\left\{y \in M:|h_\mu(f, x)-h_\mu(f, y)|<\epsilon,\left|d^{s}(x)-d^{s}(y)\right|<\epsilon,\left|d^{u}(x)-d^{u}(y)\right|<\epsilon\right\}. 
\]
There exists a countable subcovering of sets $ \left\{\Gamma^{i}\right\}_{i \in \mathbb{N}}$ 
 which covers $ \Gamma^{*} $. 

Let $ \mu^{i} $ be the conditional measure generated by $ \mu $ on $ \Gamma^{i} $.
We can apply the proof of Theorem \ref{BPS-theorem} for the ergodic case
to each measures $ \mu^{i}$, as within each $\Gamma_i$,
the entropy and dimensions are approximately constant.
This yields the desired estimates for each $\Gamma_i$.
Since $\delta$ can be chosen arbitrarily small (depending on $\epsilon$),
we extend the result to
$\mu$-almost every point in $M$.
This extension is valid due to the uniform estimates across all $\Gamma_i$.
As $\epsilon\to 0$ (and consequently $\delta\to 0$),
we can cover an arbitrarily large portion of $M$,
thus concluding the desired result for the nonergodic case.

\section{Measures maximizing $r$-neutralized entropy}\label{V-P}

Let $ f$  be a $ C^{1+\alpha} (\alpha>0)$ surface diffeomorphism and 
$ \Lambda $ a compact locally maximal hyperbolic set of $ f$  such that 
$ f|_\Lambda $ is topologically mixing. 
  We will use the method in \cite{BW03} to prove Theorem \ref{CMP}. 
Before the proof, we introduce some notions and tools. 

\subsection{Pressure function and some facts}\label{Pressure and some facts}

Given a continuous function (called potential) $\varphi: \Lambda \rightarrow \mathbb{R}$,
let $P(\varphi)$ denote its topological pressure. 
The variational principle \cite{[PW]} states that
\begin{align*}
  P(\varphi)=\sup \left\{h_\nu(f)+\int_{X} \varphi \,\rd \nu : \nu \in \mathcal{M}(f,\Lambda)\right\}.
\end{align*}
An invariant measure $\nu$ with $P(\varphi)=h_\nu(f)+\int_{\Lambda} \varphi \,\rd \nu$ is 
called an equilibrium measure for the potential $\varphi$. 
In particular, MME is the equilibrium measure for $\varphi=0$.
  For uniformly hyperbolic systems, 
the existence and the uniqueness of equilibrium measures were established for H\"older continuous potentials\cite{Bo75}.

According to \cite{Anosov}, 
the unstable and stable distributions of $(f,\Lambda)$ are H\"older continuous, 
then the functions $\phi_u: \Lambda \rightarrow \mathbb{R}$ and $\phi_s: \Lambda \rightarrow \mathbb{R}$ are also H\"older continuous, which are defined by 
\begin{align*}
  \phi_{u}(x)=\log \left\|d f | E^{u}_x\right\| \quad \text { and } \quad \phi_{s}(x)=\log \left\|d f | E^{s}_x\right\|. 
\end{align*}
For each $\nu \in \mathcal{M}(f,\Lambda)$, we introduce the following functions. 
\begin{align}\label{CMP-(8)}
 \lambda_u(\nu)=\int_{\Lambda} \phi_u \,\rd \nu , \quad \lambda_s(v)=\int_{\Lambda} \phi_s \,\rd \nu, \quad  d(\nu) = h_\nu(f)\left(\frac{1}{\lambda_u(\nu)}-\frac{1}{\lambda_s(\nu)}\right) .
\end{align}
The maps $v \mapsto \lambda_u(\nu)$ and $\nu \mapsto \lambda_s(\nu)$ are continuous 
on $\mathcal{M}(f,\Lambda)$. We define
\begin{align*}
  \lambda_u^{\min }&=\min_{\nu\in \mathcal{M}(\Lambda)} \lambda_u(\nu), \quad \lambda_u^{\max }=\max_{\nu\in \mathcal{M}(\Lambda)} \lambda_u(\nu),  \quad I_u=\left(\lambda_u^{\min }, \lambda_u^{\max }\right);\\
  \lambda_s^{\min }&=\min_{\nu\in \mathcal{M}(\Lambda)} \lambda_u(\nu), \quad \lambda_s^{\max }=\max_{\nu\in \mathcal{M}(\Lambda)} \lambda_u(\nu),   \quad I_s=\left(\lambda_s^{\min }, \lambda_s^{\max }\right) .
\end{align*}
Note that $I_u \neq \varnothing$ ($I_s \neq \varnothing$) if and only if $\phi_u$ ($\phi_s$) is not cohomologous to a constant. 
Recall that two functions $\varphi, \psi: \Lambda \rightarrow \mathbb{R}$ are said to be cohomologous
if there exists a continuous function $\eta: \Lambda \rightarrow \mathbb{R}$ such that
$\varphi-\psi=\eta-\eta \,\circ\, f$.
Futhermore, we have $P(\psi)=P(\varphi)$ if and only if $\varphi$ and $\psi$ are cohomologous.

Next, we define the function $Q: \mathbb{R}^2 \rightarrow \mathbb{R}$ by $Q(p, q)=P\left(-p \phi_u+q \phi_s\right)$. 
According to \cite{Ruelle78}, the function $Q$ is real-analytic for each $(p,q)\in \mathbb{R}^2$. 
For fixed $(p,q)$ in $\mathbb{R}^2$, we use  $\nu_{p, q} \in \mathcal{M}(f,\Lambda)$ 
to denote the unique equilibrium measure for the potential $-p \phi_u+q \phi_s$, 
which is an ergodic measure. We introduce the following notions. 
$$
  \lambda_u(p, q)=\lambda_u\left(\nu_{p, q}\right), \quad \lambda_s(p, q)=\lambda_s\left(v_{p, q}\right), \quad h(p, q)=h_{\nu_{p, q}}(f),
$$
$$
  d_u(p, q)=h(p, q) / \lambda_u(p, q), \quad d_s(p, q)=-h(p, q) / \lambda_s(p, q) .
$$
As shown in \cite{BW03}, 
functions $h, \lambda_s, \lambda_u, d_u, d_s$ are also real-analytic for each $(p,q)\in \mathbb{R}^2$.
By the variational principle of the topological pressure, we have
\begin{align*}
  Q(p,q)=h(p,q)-p\lambda_u(p,q)+q\lambda_s(p,q).
\end{align*}
There are unique nonnegative numbers $t_u$ and $t_s$ satisfying $Q(t_u,0) = Q(0,t_s) = 0$. 
 
We now list some properties of the topological pressure.  
For $\beta\in (0,1]$, 
 we denote by $C^\beta(\Lambda)$ the space of H\"older continuous functions $\varphi: \Lambda \rightarrow \mathbb{R}$ with Hölder exponent $\beta$. 
 For every $\varphi, \psi \in C^\beta(\Lambda)$ and $t \in \mathbb{R}$, we have
\begin{align}\label{B-1}
  \frac{\rd^2}{\rd t^2} P(\varphi+t \psi) \geq 0,
\end{align}
with equality if and only if $\psi$ is cohomologous to a constant, see \cite{Ruelle78} for more details.

The following propositions play a crucial role in proving Theorem \ref{CMP}. 

\begin{proposition}[\cite{BW03}]\label{CMP-Proposition 4}
  The following properties hold:
\begin{itemize}
\item[1.] For any $ p, q\in \mathbb{R}$, we have $  \lambda_u(p,q)=-\partial_p Q(p,q),\quad \lambda_s(p,q)=\partial_q Q(p,q)$.
\item[2.] If $\phi_u$ is not cohomologous to a constant and $q \in \mathbb{R}$, then:

(a) $\left\{\lambda_u(p, q): p \in \mathbb{R}\right\}=I_u$ and $\partial_p \lambda_u(p,q)<0$ for any $p\in \mathbb{R}$;

(b) $h(\cdot, 0)$ is strictly decreasing on $[0, \infty)$ and $\partial_p h(p,0)=p\partial_p\lambda_u(p,0)$;

(c) $d_u(\cdot, 0)$ is strictly increasing on $\left(-\infty, t_u\right]$ and strictly decreasing on $\left[t_u, \infty\right)$. 
More precisely, we have
\begin{align*} 
\partial_p d_u(0,0)=-\partial_p \lambda_u(0,0)\dfrac{h(0,0)}{\lambda_u(0,0)^2}>0,\,
\lim\limits_{p \rightarrow \infty} \lambda_{u}(p, q)=\lambda_{u}^{\min },\, \lim\limits_{p \rightarrow -\infty} \lambda_{u}(p, q)=\lambda_{u}^{\max }.
\end{align*}
\item[3.] If $\phi_s$ is not cohomologous to a constant and $p \in \mathbb{R}$, then 

(a)  $\left\{\lambda_s(p, q): q \in \mathbb{R}\right\}=I_s$ and $\partial_q \lambda_s(p,q)>0$ for any $q\in \mathbb{R}$;

(b) $h(0, \cdot)$ is strictly decreasing on $[0, \infty)$ and $\partial_q h(0,q)=-q\partial_q\lambda_s(0,q)$;

(c) $d_s(0, \cdot)$ is strictly increasing on $\left(-\infty, t_s\right]$, and strictly decreasing on $\left[t_s, \infty\right)$. 
More precisely, we have 
\begin{align*}
  \partial_q d_s(0,0)=\partial_q \lambda_s(0,0)\dfrac{h(0,0)}{\lambda_s(0,0)^2}>0,\,
  \lim\limits_{q \rightarrow \infty} \lambda_{s}(p, q)=\lambda_{s}^{\max},\, \lim\limits_{q \rightarrow -\infty} \lambda_{s}(p, q)=\lambda_{s}^{\min}. 
\end{align*}
\end{itemize}
\end{proposition}

\begin{proposition}[\cite{BW03}]\label{CMP-Proposition 5}
The following properties hold:
\begin{itemize}
\item[1.] For any $a \in I_u$, there exists a unique, real-analytic function 
$\gamma_u: \mathbb{R} \rightarrow \mathbb{R}$ 
such that $\lambda_u\left(\gamma_u(q), q\right)=a$  for all $q \in \mathbb{R}$;

\item[2.] For any $b \in I_s$, there exists a unique, real-analytic function 
$\gamma_s: \mathbb{R} \rightarrow \mathbb{R}$ such that 
$\lambda_s\left(p, \gamma_s(p)\right)=b$ for all $p \in \mathbb{R}$.
\end{itemize}
\end{proposition}

We use $\operatorname{dim}_H Z$ to denote the Hausdorff dimension of the set $Z$, and define the Hausdorff dimension of the measure $v$ by
$$
  \operatorname{dim}_H v=\inf \left\{\operatorname{dim}_H Z: v(\Lambda \backslash Z)=0\right\}.
$$
For surface diffeomorphisms, it was proved in \cite{Y82} that $\operatorname{dim}_H \nu=d(\nu)$ where
\[
d(\nu)=h_\nu(f)\left(\frac{1}{\lambda_u(\nu)}-\frac{1}{\lambda_s(\nu)}\right). 
\]
Thus for any given $r>0$, 
if $v$ is ergodic, we have
\begin{align}\label{CMP-(9)}
h_\nu^r(f)=h_\nu(f)+r\operatorname{dim}_H \nu=h_\nu(f)\left(1+\frac{r}{\lambda_u(\nu)}-\frac{r}{\lambda_s(\nu)}\right)=h_\nu(f)+rd(\nu).
\end{align}

For $\mu\in \mathcal{M}(f,\Lambda)$, let $\mu=\int \mu_{\tau(x)} \,\rd \mu(\tau(x))$ be its ergodic decomposition. 
By observing that the entire leaf $W^u(x)$ and $W^s(x)$ are contained in the ergodic component of $x$ (see \cite[subsection 6.2]{LY-85-annals-1}),
we conclude that for $\mu$-a.e. $x\in M$,
\begin{align}\label{ergodic-dimension}
  d^{s}_\mu(x)=d^{s}_{\mu_{\tau(x)}}(x),\quad d^{u}_\mu(x)=d^{u}_{\mu_{\tau(x)}}(x), \quad h_\mu(f,x)=h_{\mu_{\tau(x)}}(f).
\end{align}
Hence Theorem \ref{BPS-theorem-1} and \eqref{ergodic-dimension} implies
\begin{align}\label{r-neutralized-decomposition}
  \int_M h_{\mu,d}^r(f,x) \rd \mu= \int_M h_{\mu_{\tau(x)}}(f)+r\operatorname{dim}_H\mu_{\tau(x)}\,\rd \mu(\tau(x)).
\end{align}
We denote by $\mathcal{M}_E$ the set of all ergodic $f$-invariant measures on $\Lambda$. Then we have
\begin{align}\label{nonergodic-ergodic}
  \sup\left\{h_{\nu,d}^r(f) : \nu\in \mathcal{M}(f,\Lambda)\right\}=\sup\left\{h_{\nu}(f)+r\operatorname{dim}_H\nu: \nu \in \mathcal{M}_E\right\}.
\end{align}
Through \eqref{nonergodic-ergodic}, we only need to consider ergodic measures in the proof of Theorem \ref{CMP}. 
This also shows the derivation of Corollary \ref{corollary} from Theorem \ref{CMP}. 
In fact, \eqref{r-neutralized-decomposition} and \eqref{nonergodic-ergodic} 
hold for any smooth systems
where the Lyapunov exponents of any invariant measures are nonzero at almost every point.

Now we are ready to finish the proof of Theorem \ref{CMP}. 

\subsection{Proof of Theorem \ref{CMP}}\label{sec:pro}

Given $r>0$, let $\left(\nu_n\right)_{n \in \mathbb{N}}$ be a sequence of ergodic measures such that
$$
  \lim _{n \rightarrow \infty}( h_{\nu_n}(f)+r\operatorname{dim}_H \nu_n)=\sup \left\{h_{\nu}+r\operatorname{dim}_H \nu: \nu \in \mathcal{M}_E\right\}
$$
Since $\mathcal{M}(\Lambda)$ is compact in the weak* topology, 
we can assume that $\left(\nu_n\right)_{n \in \mathbb{N}}$ converges to an $f$-invariant measure $m$. 
Since the map $\nu \mapsto h_\nu(f)$ is upper semi-continuous, 
and maps $\nu \mapsto \lambda_u(\nu)$ and $\nu \mapsto \lambda_s(\nu)$ are continuous for any $\nu\in \mathcal{M}(\Lambda)$, 
we conclude that 
\begin{align}
 \sup \left\{h_{\nu}+r\operatorname{dim}_H \nu: \nu \in \mathcal{M}_E\right\}&=\lim _{n \rightarrow \infty} h_{\nu_n}(f)+r\operatorname{dim}_H \nu_n \leq h_{m}(f)+rd( m).\label{CMP-(25)}
\end{align}
To prove Theorem \ref{CMP}, it is sufficient to show that there exists $\mu \in \mathcal{M}_E$ satisfying
$$
  h_{m}(f)+rd(m)=h_{\mu}(f)+r\operatorname{dim}_H \mu .
$$

At first, we consider the most trivial case, $I_s=I_u=\varnothing$, 

\begin{lemma}\label{case 1}
  If $I_s=I_u=\varnothing$, then $\nu_{0,0}$ is the unique MM$r$NE for any $r>0$.
\end{lemma}

\begin{proof}
Through \eqref{CMP-(9)}, the ergodic MM$r$NE must be $\nu_{0,0}$.
According to \eqref{r-neutralized-decomposition}, 
we can observe that nonergodic invariant measures do not 
achieve the supremum of the $r$-neutralized entropy.
Thus, we complete the proof of this lemma.
\end{proof}

For the case that the MME is the MM$r$NE for some $r>0$, 
see the Proposition \ref{varnothing} for more information. 

Next, we introduce the following criterion lemma. 

\begin{lemma}\label{CMP-Lemma 3}
  If $p, q \in \mathbb{R}$ with $\lambda_u(p, q)=\lambda_u(m)$ and $\lambda_s(p, q)=\lambda_s(m)$, then $m=\nu_{p, q}$.
\end{lemma}

\begin{proof}
  By the definition of the equilibrium measure $\nu_{p,q}$, we have
  \begin{align*}
    h(p, q)+\int_{\Lambda}\left(-p \phi_u+q \phi_s\right) d \nu_{p, q} & =h(p, q)-p \lambda_u(m)+q \lambda_s(m)                         \\
                                                                       & \geq h_m(f)+\int_{\Lambda}\left(-p \phi_u+q \phi_s\right) d m. 
  \end{align*}
  Thus we obtain $h(p, q) \geq h_m(f)$ and the equality holds if and only if $\nu_{p, q}=m$.
  On the other hand, by combining \eqref{CMP-(9)} and \eqref{CMP-(25)}, we derive $h(p, q) \leq h_m(f)$. 
  Then we conclude $h(p, q)=h_m(f)$ and $m=\nu_{p, q}$.
\end{proof}

By applying Lemma \ref{CMP-Lemma 3}, 
the following two lemmas provide the proof of Theorem \ref{CMP} under specific conditions. 

\begin{lemma}\label{CMP-Lemma 1}
If $\lambda_s(m) \in I_s$, then there exist $p, q\in \mathbb{R}$
such that $\lambda_u\left(p, q\right)=\lambda_u(m)$ and $\lambda_s\left(p, q\right)=\lambda_s(m)$. 
Furthermore, we have $m=\nu_{p,q}$. 
\end{lemma}

\begin{proof}
By Proposition \ref{CMP-Proposition 5}, there exists a unique function $\gamma_s:\mathbb{R}\to \mathbb{R}$ such that 
\begin{align}\label{gamma-s}
\lambda_s(t, \gamma_s(t))=\lambda_s(m),\quad \text{ for all } t\in \mathbb{R}. 
\end{align}
For any $p \in \mathbb{R}$, since $v_{p, \gamma_s(p)}$ is the equilibrium measure 
of the potential $-p \phi_u+\gamma_s(p) \phi_s$ and considering \eqref{gamma-s}, we have 
\begin{align}\label{CMP-(27)}
  h\left(p, \gamma_s(p)\right)-p \lambda_u\left(p, \gamma_s(p)\right) \geq h_m(f)-p \lambda_u(m). 
\end{align}
By using $\lambda_u\left(p, \gamma_s(p)\right)>0$ holds for any $p\in \mathbb{R}$, we obtain
\begin{align}\label{CMP-(28)}
  \frac{h\left(p, \gamma_s(p)\right)}{\lambda_u\left(p, \gamma_s(p)\right)}-\frac{h_m(f)}{\lambda_u(m)} \geq\left(1-\frac{\lambda_u(m)}{\lambda_u\left(p, \gamma_s(p)\right)}\right)\left(p-\frac{h_m(f)}{\lambda_u(m)}\right) .
\end{align}
Let $\kappa=\dfrac{h_m(f)}{ \lambda_u(m)}$. 
By choosing $p=\kappa$ and \eqref{CMP-(28)}, we conclude
\begin{align*}
 \frac{ h\left(\kappa, \gamma_s(\kappa)\right) }{ \lambda_u\left(\kappa, \gamma_s(\kappa)\right) }\geq \frac{h_m(f)}{ \lambda_u(m)} .
\end{align*}
If the inequality $\lambda_u\left(\kappa, \gamma_s(\kappa)\right)>\lambda_u(m)$ holds, we have $h\left(\kappa, \gamma_s(\kappa)\right)>h_m(f)$ which contradicts \eqref{CMP-(25)} and \eqref{gamma-s}. 
Thus, we obtain
\begin{align}\label{CMP-(30)}
  \lambda_u\left(\kappa, \gamma_s(\kappa)\right) \leq \lambda_u(m) .
\end{align}
On the other hand, by setting $p=0$ in \eqref{CMP-(27)},  we derive that $h\left(0, \gamma_s(0)\right) \geq h_m(f)$. 
Additionally, by using \eqref{CMP-(9)} and \eqref{CMP-(25)}, we have
\begin{align}\label{CMP-(32)}
  \lambda_u\left(0, \gamma_s(0)\right) \geq \lambda_u(m) .
\end{align}
Combining \eqref{CMP-(30)}, \eqref{CMP-(32)} and the continuity of $p \mapsto \lambda_u\left(p, \gamma_s(p)\right)$, 
 there exists $p \in[0, \kappa]$ such that $\lambda_u\left(p, \gamma_s(p)\right)=\lambda_u(m)$. 
By setting $q=\gamma_s(p)$, we finish the proof of this lemma. 
\end{proof}

Similarly, we  obtain the following lemma, of which the proof is left to the reader. 

\begin{lemma}\label{CMP-Lemma 2}
  If $\lambda_u(m) \in I_u$, then there exist $p, q\in \mathbb{R}$
  such that $\lambda_u\left(p, q\right)=\lambda_u(m)$ and $\lambda_s\left(p, q\right)=\lambda_s(m)$.
  Furthermore, we have $m=\nu_{p,q}$.
\end{lemma}

Therefore, we are left with the case where $\lambda_s(m) \notin I_s$ and $\lambda_u(m) \notin I_u$. 

\begin{lemma}\label{CMP-Lemma 7}
 If $\lambda_u(m) \notin I_u$ and $\lambda_s(m) \notin I_s$, then 
 $ \lambda_{u}(\mu)=\lambda_{u}^{\min } $ and $ \lambda_{s}(\mu)=\lambda_{s}^{\max } $. 
Moreover, there exists $v \in \mathcal{M}_E$ such that 
\[
\lambda_u(v)=\lambda_u(m),\quad \lambda_s(v)=\lambda_s(m),\quad \text{ and }\quad h_v(f)=h_m(f).
\]
\end{lemma}

\begin{proof}
Firstly, we prove that only one case $\lambda_u(m)=\lambda_u^{\min }$ and $\lambda_s(m)=\lambda_s^{\max }$ can occur. 
  
  \begin{itemize}
\item{\bf Claim 1.} The case where $I_u=\varnothing$,  $I_s \neq \varnothing$ and $\lambda_s(m)=\lambda_s^{\min }$ will not occur.

By using $I_u=\varnothing$, we have $\lambda_u(0,0)=\lambda_u(m)$ where 
 $\nu_{0,0}$ is the MME with $h(0,0) \geq h_m(f)$. 
Thus \eqref{CMP-(9)} implies 
\begin{align*}
  h_{\nu_{0,0}}(f)+\left(1+\frac{r}{\lambda_u(\nu_{0,0})}-\frac{r}{\lambda_s(\nu_{0,0})}\right)
  >h_m(f)\left(1+\frac{r}{\lambda_u(m)}-\frac{r}{\lambda_s(m)}\right)
\end{align*}
which contradicts \eqref{CMP-(25)} and \eqref{CMP-(9)}. 

\item {\bf Claim 2.} The case where $I_s=\varnothing$, $I_u \neq \varnothing$ and $\lambda_u(m)=\lambda_u^{\max }$ will not occur.

Similar to the proof of {\bf Claim 1}.

Then the only remaining situation we need to consider is the case $I_u \neq \varnothing$ and $I_s \neq \varnothing$, 
  in this case $\lambda_u(m) \in \left\{\lambda_u^{\min }, \lambda_u^{\max }\right\}$ and $\lambda_s(m) \in \left\{\lambda_s^{\min }, \lambda_s^{\max }\right\}$. 
  Under this situation, we will proceed with a classified discussion.

  \item {\bf Claim 3.} The case where $\lambda_u(m)=\lambda_u^{\max }$ and $ \lambda_s(m)=\lambda_s^{\min }$ will not occur. 

According to $h(0,0) \geq h_m(f)$ and Proposition \ref{CMP-Proposition 4}, 
we have $\lambda_u(0,0)<\lambda_u(m)$ and $\lambda_s(0,0)>\lambda_s(m)$ which contradict \eqref{CMP-(9)} and \eqref{CMP-(25)}. 

\item{\bf Claim 4.} The case where $\lambda_u(m)=\lambda_u^{\min }$ and 
$ \lambda_s(m)=\lambda_s^{\min }$ will not occur.

By applying $\lambda_u(m)=\lambda_u^{\min }$ and $v_{p, 0}$ is the equilibrium measure of $-p \phi_u$, 
we obtain $\lambda_u(m)\leq \lambda_u(p,0)$ and $h(p, 0)-p \lambda_u(p, 0)\geq h_m(f)-p \lambda_u(m)$. 
Therefore we conclude that $h(p, 0)>h_m(f)$ for all $p>0$. 
By $\lambda_s(m)\leq \lambda_s(p,0)$ and $h(p, 0)>h_m(f)$ for $p>0$, we have
\begin{align}\label{claim 4-1}
  d_s(p,0)=\frac{h(p,0)}{\lambda_s(p,0)} > \frac{h_m(f)}{\lambda_s(m)}.
\end{align}
Fix $\epsilon>0$, by applying Proposition \ref{CMP-Proposition 4}, there exists a real number $p_0$ such that $|\lambda_u(p,0)-\lambda_u(m)|<\epsilon$ for any $p\geq p_0$. 
Thus for $p\geq p_0$, we have
\begin{align*}
  d_u(p,0)=\frac{h(p,0)}{\lambda_u(p,0)}>\frac{h_m(f)}{\lambda_u(m)+\epsilon}=\frac{h_m(f)}{\lambda_u(m)}-\frac{\epsilon h_m(f)}{\lambda_u(m)(\lambda_u(m)+\epsilon)}. 
\end{align*}
Since $h_m(f)$ and $\lambda_u(m)$ are finite, we conclude that 
\begin{align}\label{claim 4-2}
d_u(p,0)\geq \dfrac{h_m(f)}{\lambda_u(m)} \quad \text{ holds for sufficiently large } p. 
\end{align}
By \eqref{claim 4-1} and \eqref{claim 4-2}, for all sufficiently large $p$, we obtain that
\begin{align*}
 h_{\nu_{p, 0}}(f)+r\operatorname{dim}_H \nu_{p, 0}&=h_{\nu_{p, 0}}(f)+rd_u(p, 0)+rd_s(p, 0) \\
 &\geq h_{\nu_{p, 0}}(f)+ r\frac{h_m(f)}{\lambda_u(m)}-r\frac{h(p, 0)}{\lambda_s(p, 0)}\\
 &>h_{m}(f)+rd(m)
\end{align*}
 which contradicts \eqref{CMP-(25)}. Thus, we complete the proof of this claim.

\item {\bf Claim 5.} The case where $\lambda_u(m)=\lambda_u^{\max }$ and $\lambda_s(m)=\lambda_s^{\max }$ will not occur.

Similar to the proof of {\bf Claim 4}.
\end{itemize}
Thus we conclude, only the case where $\lambda_u(m)=\lambda_u^{\min }$ and $\lambda_s(m)=\lambda_s^{\max }$ could occur.

Next, we consider an ergodic decomposition $\tau$ of $m$, it is a probability measure on the metrizable space $\mathcal{M}(\Lambda)$ with $\tau\left(\mathcal{M}_E\right)=1$ such that for all $\varphi \in C(\Lambda, \mathbb{R})$, 
\begin{align}\label{CMP-(38)}
  \int_{\mathcal{M}} \int_{\Lambda} \varphi \,\rd \nu \,\rd \,\tau(\nu)=\int_{\Lambda} \varphi \,\rd m.
\end{align}
Since the function $\phi_u$ is H\"older continuous, we obtain
$$
  \lambda_u^{\min }=\lambda_u(m)=\int_{\mathcal{M}} \lambda_u(v) \,\rd \tau(\nu) .
$$
Since $\lambda_u(\nu) \geq \lambda_u^{\min }$ for all $\nu \in \mathcal{M}(\Lambda)$, 
there exists a subset $A_1 \subset \mathcal{M}_E$ such that 
\begin{align*}
\tau\left(A_1\right)=1, \quad \text{ and }\quad  \lambda_u(v)=\lambda_u^{\min } \quad \text{ for all } \nu \in A_1. 
\end{align*}
Similarly, there exists a subset $A_2 \subset \mathcal{M}_E$ such that 
\begin{align*}
\tau\left(A_2\right)=1, \quad \text{ and } \quad \lambda_s(v)=\lambda_s^{\max } \quad \text{ for all } \nu \in A_2.
\end{align*}
By \eqref{CMP-(9)} and \eqref{CMP-(25)}, we obtain $h_v(f) \leq h_m(f)$ for all $\nu \in A_1 \cap A_2$. 
Since $\tau\left(A_1 \cap A_2\right)=1$ and $h_m(f)=\int_{\mathcal{M}} h_\nu(f) \,\rd \tau(v)$, 
there exists $A \subset A_1 \cap A_2$ with $\tau(A)=1$ such that $h_\nu(f)=h_m(f)$ for all $\nu \in A$. 
This concludes the proof of the lemma.
\end{proof}

Next, we analyze the behavior of  measures $\mu_{r_n}$ 
as $ r_n $ approaches to $0$ or infinity. 

{\bf Case 1: $r_n\to 0$. }

Since $\nu_{0,0}$ is the MME, we immediately obtain
\begin{align*}
  h_{\nu_{0,0}}(f)+2r\geq h_{\mu_{r_n}}(f)+r_n\operatorname{dim}_H \mu_{r_n}\geq h_{\nu_{0,0}}(f)+r_n\operatorname{dim}_H \mu_{r_n}\geq h_{\nu_{0,0}}(f). 
\end{align*}
As $n\to \infty$, we derive $ \lim\limits_{n\to \infty} h_{\mu_{r_n}}(f)=h_{\nu_{0,0}}(f)$. 
Based on the upper semi-continuity of the entropy map $\mu\to h_\mu(f)$
and the uniqueness of the MME, 
we conclude that any limit point of $\mu_{r_n}$ is $\nu_{0,0}$. 
Therefore, we obtain $\lim\limits_{n\to \infty}\mu_{r_n}=\nu_{0,0}$. 

In fact, besides the entropy of measures $\mu_{r_n}$ approaching the entropy of $\nu(0,0)$,
 there are similar results for Lyapunov exponents. 
 According to \cite{Kad}, there exists a constant $c>0$ such that 
 $|\lambda_u(\nu(0,0))-\lambda_u(\mu)|\leq c|h_{\nu(0,0)}(f)-h_\mu(f)|$ 
 for any ergodic measure $\mu$. 

{\bf Case 2: $r_n\to\infty$. }

Clearly, the topological entropy $h_{\text{top}}(f)$ of $(f|_\Lambda,\Lambda)$ is finite.
We have
\begin{align*}
  \frac{h_{\text{top}}(f)}{r_n}+\sup\{\operatorname{dim}_H \nu: \nu \in \mathcal{M}(\Lambda) \}&\geq \frac{h_{\mu_{r_n}}(f)}{r_n}+\operatorname{dim}_H \mu_{r_n}\geq \sup\{\operatorname{dim}_H \nu: \nu \in \mathcal{M}(\Lambda) \}.
\end{align*}
We suppose that $\mu_{r_n}$ converges to $\mu'$ as $n\to\infty$
(or we take a convergent subsequence if necessary). 
By letting $n\to\infty$, we conclude that 
$\operatorname{dim}_H \mu'=\sup\{\operatorname{dim}_H \nu: \nu \in \mathcal{M}(\Lambda) \}$. 
Then any limit point of $\mu_{r_n}$ is MMHD for the hyperbolic sets. 

This concludes the proof of Theorem \ref{CMP}. 

Based on our previous discussion, MM$r$NE $\mu_r$ is an equilibrium measure, 
except in the case $\lambda_u(\mu_r) \notin I_u$ and $\lambda_s(\mu_r) \notin I_s$, 
as noted in Remark \ref{non-equilibrium}. 
However, in this case, 
using the method in \cite[Lemma 8]{BW03}, 
we have the following result. 

\begin{proposition}
Given $r>0$, if $ \mu$ is an ergodic measure maximizing $r$-neutralized entropy 
such that 
$\lambda_u(\mu) \notin I_u$ and $\lambda_s(\mu) \notin I_s$ then
\begin{align*}
\lim_{t\to\infty} h^r_{\nu_{t,t}}(f)=h^r_\mu(f)=\sup\left\{h_{\nu,d}^{r}(f): \nu \in \mathcal{M}(f,X)\right\}.
\end{align*}
\end{proposition}

\begin{proof}
We will show that as $t\to\infty$, the Lyapunov exponents of $\nu_{t,t}$ converge 
to those of $\mu$, which then implies the convergence of $r$-neutralized entropies.

 Since $ v_{t, t} $ is the equilibrium measure of $-t \phi_{u}+t \phi_{s} $, we have
\begin{align}\label{CMP-(41)}
h(t, t)-t \lambda_{u}(t, t)+t \lambda_{s}(t, t) \geq h_{\mu}(f)-t \lambda_{u}(\mu)+t \lambda_{s}(\mu) .
\end{align}
By applying Proposition \ref{CMP-Proposition 4} and Lemma \ref{CMP-Lemma 7}, we obtain
\begin{align}\label{CMP-(42)}
\lambda_{u}(t, t)>\lambda_{u}(\mu)=\lambda_{u}^{\min } \quad \text { and } \quad \lambda_{s}(t, t)<\lambda_{s}(\mu)=\lambda_{s}^{\max}. 
\end{align}
Combining \eqref{CMP-(41)} and \eqref{CMP-(42)}, we can conclude that 
$h(t, t)>h_{\mu}(f)$ for all $t \geq 0$.

Fix $ \epsilon >0 $. For any $ t> \epsilon^{-1} h_{\text {top }}(f)   $,
by using \eqref{CMP-(41)}, we derive
\begin{align*}
  \lambda_{u}(t, t)-\lambda_{s}(t, t) & \leq \frac{h(t, t)-h_{\mu}(f)}{t}+\lambda_{u}(\mu)-\lambda_{s}(\mu) \\
                                      & <\frac{\epsilon  h(t, t)}{h_{\text{top}}(f)}+\lambda_{u}(\mu)-\lambda_{s}(\mu) \leq \epsilon +\lambda_{u}(\mu)-\lambda_{s}(\mu) .
\end{align*}
This inequality, combined with \eqref{CMP-(42)}, implies that as $t\to\infty$, we have 
\begin{align}\label{pro-1}
\lim_{t\to\infty}\lambda_{u}(t, t) =\lambda_{u}(\mu), \quad 
\lim_{t\to\infty}\lambda_{s}(t, t)= \lambda_{s}(\mu). 
\end{align}
Hence, by using \eqref{CMP-(9)}, we have 
\begin{align*}
\lim_{t\to\infty} h^r_{\nu_{t,t}}(f)
&=\lim_{t\to\infty} h(t,t)\left(1+\frac{r}{\lambda_u(t,t)}-\frac{r}{\lambda_s(t,t)}\right)\\
&\geq h_\mu(f) \lim_{t\to\infty}\left(1+\frac{r}{\lambda_u(t,t)}-\frac{r}{\lambda_s(t,t)}\right)\\
&=h_\mu(f) \left(1+\frac{r}{\lambda_u(\mu)}-\frac{r}{\lambda_s(\mu)}\right)=h_\mu^r(f). 
\end{align*}

On the other hand, since $\mu$ is the measure maximizing the $r$-neutralized entropy, 
we have $h^r_{\nu_{t,t}}(f)\leq h^r_\mu(f)$ for any $t>0$.

Therefore, we conclude $\lim\limits_{t\to\infty} h^r_{\nu_{t,t}}(f) = h^r_\mu(f)$, 
which completes the proof.

\end{proof}

\subsection{Non-uniqueness of ergodic MM$r$NEs}\label{multiple subsection}

In this subsection, we construct a two-dimensional linear horseshoe 
that exhibits multiple ergodic MM$r$NEs for large $r$. 
This construction builds upon a method used in \cite{Rams}. 

Fix positive numbers $\eta_1$ and $\eta_2$ satisfying $\eta_1+\eta_2<1$. 
We define two affine maps as follows. 
\begin{align*}
f_1: &[0,\eta_1] \times [0,1] \to [0,1] \times [0,\eta_2], \quad (x,y)  \mapsto  (\eta_1^{-1}x,\eta_2 y),  \\
f_2: &[1-\eta_2,1] \times [0,1] \to [0,1] \times [1-\eta_1,1],\quad  (x,y) \mapsto (\eta_2^{-1}(x-1+\eta_2), 1-\eta_1+\eta_1 y) .
\end{align*}
It is easy to see that, any point $(x, y) \in [0,1]^2$ belongs to the domain 
of at most one map $f_i$,
and similarly, to the preimage of at most one map $f_j$. 
This allows us to define  the maps $f$ and $f^{-1}$ on a subset of $[0,1]^2$,
where $f$ is piecewise linear. Specifically, $f$ restricted to the domain of $f_1$
is exactly $f_1$, and $f$ restricted to the domain of $f_2$ is exactly $f_2$.
Via a direct calculation, we obtain
\begin{align}
  df(x, y)=\left(\begin{matrix} \eta_1^{-1} & 0 \\ 0 & \eta_2 \end{matrix}\right),\quad \text{ for any } (x,y)\in C_1,\label{LY-1} \\
  df(x, y)=\left(\begin{matrix}\eta_2^{-1} & 0 \\ 0 & \eta_1 \end{matrix}\right),\quad \text{ for any } (x,y)\in C_2.\label{LY-2}
\end{align}
Let $\Lambda$ be the set of points for which both $f^n$ and
$f^{-n}$ exist for all integers $n$. The set $\Lambda$ is a compact hyperbolic set.
We introduce sets $C_1$ and $C_2$ as follows.
\begin{align*}
  C_1 & :=\{(x,y)\in \Lambda : (x,y)\in [0,\eta_1] \times[0,1]\},   \\
  C_2 & :=\{(x,y)\in \Lambda : (x,y)\in [1-\eta_2,1] \times[0,1]\}.
\end{align*}
Note that $\Lambda=C_1\cup C_2$ and $C_1\cap C_2=\varnothing$.
This set $\Lambda$ is called a Smale horseshoe or linear horseshoe, due to the map $f$ is linear on $C_1$ and $C_2$. 
Now let us introduce some properties of this horseshoe. 

Firstly, as shown in \cite[Section 1.5]{Smale}, 
the map $(f,\Lambda)$ can be extended to a global diffeomorphism on $\mathbb{S}^2$. 
This extension allows us to directly apply the results obtained before, 
such as Theorem \ref{BPS-theorem-1}, to $(f,\Lambda)$. 

By \cite[Section 4.5]{Wen}, the set $\Lambda$ is a locally maximal hyperbolic set. 
Moreover, we have the following theorem. 

\begin{theorem}[\cite{BP-nonuniform}]
Any two-dimensional horseshoe is a hyperbolic Cantor set, 
which is topologically conjugate to the shift map on $\{1,2\}^{\mathbb{Z}}$. 
\end{theorem}

As stated in the above theorem, it follows that $f|_\Lambda$ is topological mixing. 
This property, combined with the locally maximal hyperbolic set, 
allows us to apply Theorem \ref{CMP}, 
which guarantees the existence of an ergodic MM$r$NE of
$(f,\Lambda)$ for any positive $r$. 
We refer the reader to \cite{Smale,BP-nonuniform} for more information about the linear horseshoe.

We continue with an auxiliary lemma in \cite{Rams}. 

\begin{lemma}[\cite{Rams}]\label{Rams-lemma}
  The metric entropy of ergodic invariant measure $\nu$ is not greater than 
  the metric entropy of the Bernoulli measure defined by the probability vector 
  $(\nu(C_1), \nu(C_2))$, with equality if and only if $\nu$ is Bernoulli.
\end{lemma}

Applying the above lemma, we obtain the following 

\begin{lemma}\label{Bernoulli}
For $r>0$, an ergodic $f$-invariant measure is MM$r$NE if and only if 
it is a Bernoulli measure. 
\end{lemma}

\begin{proof}
Firstly, for a fixed ergodic $f$-invariant measure $\mu$, 
we compute the Lyapunov exponents of the measure $\mu$. 
Let $\chi_{C_1}$ and $\chi_{C_2}$ be the characteristic functions of $C_1$ and $C_2$ respectively.  
The Birkhoff ergodic theorem implies that 
for $\mu$-a.e. $x\in \Lambda$, 
\begin{align}\label{LY-3}
\lim_{n\to \infty}\frac{1}{n}\sum_{i=0}^{n-1} \chi_{C_1}(f^i x)=\mu(C_1), 
\quad \lim_{n\to \infty}\frac{1}{n}\sum_{i=0}^{n-1} \chi_{C_2}(f^i x)=\mu(C_2). 
\end{align}
From \eqref{LY-1}, \eqref{LY-2}, \eqref{LY-3} and the definition of Lyapunov exponents (see \eqref{LY-0}), 
we can derive that 
the Lyapunov exponents of the ergodic $f$-invariant measure $\mu$ are 
\begin{align}\label{horseshoe Lyapunov exponents}
  \lambda_1(\mu) =-\mu\left(C_1\right) \log \eta_1-\mu\left(C_2\right) \log \eta_2,  \quad   \lambda_2(\mu)  =\mu\left(C_1\right) \log \eta_{2}+\mu\left(C_2\right) \log \eta_{1}.
\end{align}

Given $r>0$, let $\nu$ be an ergodic MM$r$NE. Let $\mu_{\nu(C_1)}$ be the Bernoulli measure defined by the probability vector $(\nu(C_1), \nu(C_2))$.
Clearly, from \eqref{horseshoe Lyapunov exponents}, $\nu$ and $\mu_{\nu(C_1)}$ have the same Lyapunov exponents. 
By Lemma \ref{Rams-lemma}, 
we have $h_\nu(f)\leq h_{\mu_{\nu(C_1)}}(f)$. 
Consequently,
\begin{align*}
  h_\nu^r(f)&=h_\nu(f)\left(1+\frac{r}{\lambda_1(\nu)}-\frac{r}{\lambda_2(\nu)}\right)\\
  &\leq h_{\mu_{\nu(C_1)}}(f)\left(1+\frac{r}{\lambda_1(\mu_{\nu(C_1)})}-\frac{r}{\lambda_2(\mu_{\nu(C_1)})}\right)=h^r_{\mu_{\nu(C_1)}}(f), 
\end{align*}
where equality holds if and only if $\nu=\mu_{\nu(C_1)}$. 
Hence if $\nu$ is a MM$r$NE, 
we can conclude that $\nu = \mu_{\nu(C_1)}$. 
Recall that $\mu_{\nu(C_1)}$ is a Bernoulli measure, 
thus $\nu$ is indeed a Bernoulli measure. 
This completes the proof of this lemma.
\end{proof}

Lemma \ref{Bernoulli} significantly simplifies our study of ergodic MM$r$NEs
by focusing only on Bernoulli measures. 
With the help of Lemma \ref{Bernoulli}, we are ready to obtain the  main result in this subsection. 

\begin{theorem}\label{multiple}
  There exists a two-dimensional linear hourseshoe with multiple ergodic measures maximizing $r$-neutralized entropy for any $r\geq 3$.
\end{theorem}

\begin{proof}
 By Lemma \ref{Bernoulli}, we only need to consider the Bernoulli measures of $(f,\Lambda)$. 
  A Bernoulli measure $\mu$ is uniquely determined by
  a probability vector $ (p, 1-p) $, where $\mu(C_1)=p$, $\mu(C_2)=1-p$, and $p\in [0,1]$
  (note that $0\log 0=0$). 
  To emphasize the dependence on $p$, we use measure $\mu_{p}$ to represent measure $\mu$.
  For simplicity, we will use the following notations,
  \begin{align*}
    h(p)         & :=h(\mu_{p})=-p \log p-(1-p)\log (1-p),                                                    \\
    \lambda_1(p) & :=\lambda_1(\mu_{p})=-p \log \eta_1-(1-p)\log \eta_2,                                      \\
    \lambda_2(p) & := \lambda_2(\mu_{p})=p \log \eta_2+(1-p)\log \eta_1 ,                                    \\
    h^r(p)       & :=h^r_{\mu_{p}}(f)=h(p)+r\left(\frac{h(p)}{\lambda_1(p)}-\frac{h(p)}{\lambda_2(p)}\right).
  \end{align*}
Here, the first formula is exactly the well known metric entropy formula, the second one and third one are due to \eqref{horseshoe Lyapunov exponents}, and the last one is by Theorem \ref{BPS-theorem-1}. 

We would like to understand when $h^r(p)$ is a locally maximal value (thus the corresponding measure is potentially a MM$r$NE.).

By direct calculations, we have
  \begin{align*}
    \frac{\rd}{\rd p} h(p)                                         & =-\log p+\log (1-p), \\
     \frac{\rd}{\rd p} \lambda_1(p)&=\frac{\rd}{\rd p} \lambda_2(p)=-\log \eta_1+\log \eta_2,   \\
  \frac{\rd}{\rd p} \left(\frac{h}{\lambda_1}\right)(p)          & =\frac{-\log (1-p) \log \eta_{1}+\log p \log \eta_{2}}{\left(p \log \eta_{1}+\left(1-p\right) \log \eta_{2}\right)^{2}},               \\ 
    \frac{\rd}{\rd p} \left(-\frac{h}{\lambda_2}\right)(p)         & =\frac{-\log (1-p) \log \eta_{2}+\log p \log \eta_{1}}{\left(p \log \eta_{2}+\left(1-p\right) \log \eta_{1}\right)^{2}}. 
   \end{align*}
  The second derivatives are
   \begin{align*}
    \frac{\rd^2}{\rd p^2} h(p)&=-\frac{1}{p}-\frac{1}{1-p} ,\quad \quad \frac{\rd^2}{\rd p^2} \lambda_1(p)=\frac{\rd^2}{\rd p^2} \lambda_2(p)=0,\\
    \frac{\rd^{2}}{\rd p^{2}} \left(\frac{h}{\lambda_1}\right)(p)  & =  \frac{\left((1-p)^{-1} \log \eta_{1}+p^{-1} \log \eta_{2}\right)\left(p \log \eta_{1}+\left(1-p\right) \log \eta_{2}\right)}{\left(p \log \eta_{1}+\left(1-p\right) \log \eta_{2}\right)^{3}}                                                          \\
                                                                   & +\frac{2\left(\log \eta_{1}-\log \eta_{2}\right)\left(\log \left(1-p\right) \log \eta_{1}-\log p \log \eta_{2}\right)}{\left(p \log \eta_{1}+\left(1-p\right) \log \eta_{2}\right)^{3}} ,                                                                 \\
    \frac{\rd^{2}}{\rd p^{2}} \left(-\frac{h}{\lambda_2}\right)(p) & =  \frac{\left((1-p)^{-1} \log \eta_{2}+p^{-1} \log \eta_{1}\right)\left(p \log \eta_{2}+\left(1-p\right) \log \eta_{1}\right)}{\left(p \log \eta_{2}+\left(1-p\right) \log \eta_{1}\right)^{3}}                                                          \\
                                                                   & +\frac{2\left(\log \eta_{2}-\log \eta_{1}\right)\left(\log \left(1-p\right) \log \eta_{2}-\log p \log \eta_{1}\right)}{\left(p \log \eta_{2}+\left(1-p\right) \log \eta_{1}\right)^{3}}.
  \end{align*}
  Hence, we obtain
  \begin{align*}
    \frac{\rd}{\rd p} h^r(\frac{1}{2})         & = \frac{\rd}{\rd p} h(\frac{1}{2})+r\left(\frac{\rd}{\rd p} \left(\frac{h}{\lambda_1}\right)(\frac{1}{2})+\frac{\rd}{\rd p} \left(-\frac{h}{\lambda_2}\right)(\frac{1}{2})  \right)  =0,           \\
    \frac{\rd^{2}}{\rd p^{2}} h^r(\frac{1}{2}) & =\frac{\rd^{2}}{\rd p^{2}} h(\frac{1}{2})+r\left( \frac{\rd^2}{\rd p^2} \left(\frac{h}{\lambda_1}\right)(\frac{1}{2})+\frac{\rd^2}{\rd p^2} \left(-\frac{h}{\lambda_2}\right)(\frac{1}{2}) \right) \\
                                               & =-4+16r\frac{\left( \log \eta_{1}+ \log \eta_{2}\right)^2-2\log 2 \left(\log \eta_{1}-\log \eta_{2}\right)^2}{\left(\log \eta_{1}+ \log \eta_{2}\right)^{3}}.
  \end{align*}
  
Set $\eta_1=0.9703$ and $\eta_2=\eta_1^{117}$. By calculation, they satisfy the initial assumption $\eta_1+\eta_2<1$, and  for $r\geq 3$, 
  \begin{align*}
    \frac{\rd}{\rd p} h^r(\frac{1}{2})=0,\quad \frac{\rd^{2}}{\rd p^{2}} h^r(\frac{1}{2})>0.
  \end{align*}
  This implies that $p=\frac{1}{2}$ is not a local maximum of $h^r(p)$.
  Now notice that, for any $p\in [0,1]$ we have $h^r(p)=h^r(1-p)$, which implies $\mu_{1-p}$ is a MM$r$NE if $\mu_p$ is.
  As a result, $f$ has at least two MM$r$NEs, completing the proof.
\end{proof}

\begin{remark}
The specific values $\eta_1=0.9703$ and $\eta_2=\eta_1^{117}$ provide a lower bound for $r$. 
While these parameters could be optimized to yield a smaller lower bound,
 this will not lead to a significant improvement. 
\end{remark}


\subsection{Rigidity results}\label{futher discussion}

In this subsection, we prove some rigidity results for MM$r$NE. 

For convenience, we denote the measure of maximal entropy as the measure maximizing $0$-neutralized 
entropy (MM$0$NE) 
and the measure of maximal Hausdorff dimension as the measure maximizing$\infty$-neutralized entropy (MM$\infty$NE). 

\begin{theorem}\label{rigidity results}
  Let $ f$  be a $ C^{1+\alpha} $ surface diffeomorphism and $ \Lambda $ a compact locally maximal hyperbolic set of $ f$  such that $ f|_\Lambda $ is topologically mixing. 
Given $0\leq r_0<r_1\leq \infty$, 
if the ergodic measure $\mu$ maximizing $r_1$-neutralized entropy is also the measure maximizing $r_0$-neutralized entropy,
then $\mu$ is the measure maximizing $r$-neutralized entropy for any $r_0<r<r_1$. 
\end{theorem}

\begin{proof}
Firstly, we prove the case of $0<r_0<r_1<\infty$. 
According to \eqref{nonergodic-ergodic}, we have
\begin{align*}
  h_{\mu}(f)+r_0\operatorname{dim}_H\mu & = \sup\left\{h_{\nu}(f)+r_0\operatorname{dim}_H\nu: \nu \in \mathcal{M}_E\right\},\nonumber                   \\
  h_{\mu}(f)+r_1\operatorname{dim}_H\mu & = \sup\left\{h_{\nu}(f)+r_1\operatorname{dim}_H\nu: \nu \in \mathcal{M}_E\right\}. 
\end{align*}
Fix $r\in (r_0, r_1)$, there exists a constant $\kappa\in (0,1)$ such that $r=\kappa r_0+(1-\kappa) r_1$. Then
\begin{align*}
&h_{\mu}(f)+r\operatorname{dim}_H\mu\\
=&\kappa (h_{\mu}(f)+r_0\operatorname{dim}_H\mu)+(1-\kappa)(h_{\mu}(f)+r_1\operatorname{dim}_H\mu)\\
=&\kappa \sup\left\{h_{\nu}(f)+r_0\operatorname{dim}_H\nu: \nu \in \mathcal{M}_E\right\}+(1-\kappa) \sup\left\{h_{\nu}(f)+r_1\operatorname{dim}_H\nu: \nu \in \mathcal{M}_E\right\}\\
\geq & \sup\left\{h_{\nu}(f)+r\operatorname{dim}_H\nu: \nu \in \mathcal{M}_E\right\}. 
\end{align*}
Hence the proof of this case is complete. Clearly, we have 
\begin{align*}
&\text{If $\mu$ is the MME}, \quad \text{then } h_\mu(f)=\sup\left\{h_{\nu}(f): \nu \in \mathcal{M}_E\right\},\\
&\text{If $\mu$ is the ergodic MMHD}, \quad \text{then } \operatorname{dim}_H\mu=\sup\left\{\operatorname{dim}_H\nu: \nu \in \mathcal{M}_E\right\}.
\end{align*}
The proofs of other cases are entirely analogous.
\end{proof}

Lemma \ref{case 1} states that the MME $\nu_{0,0}$ is the unique MM$r$NE for any positive $r$ 
if $I_s=I_u=\varnothing$. 
The following theorem studies the converse statement. 

\begin{theorem}\label{varnothing}
  Let $ f$  be a $ C^{1+\alpha} $ area-preserving surface diffeomorphism and $ \Lambda $ a compact locally maximal hyperbolic set of $ f$  such that $ f|_\Lambda $ is topologically mixing. 
  If the measure of maximal entropy $\nu$ is the measure maximizing $r$-neutralized entropy for some $0<r\leq \infty$, 
  then the following dichotomy holds: 
  \begin{itemize}
 \item[(1)] $\phi_u$ and $\phi_s$ are cohomologous to constants, 
 \item[(2)] Neither $\phi_u$ nor $\phi_s$ is cohomologous to a constant but $-\phi_u+\alpha\phi_s$ is cohomologous to a constant, 
 where $\alpha=\lambda_u(0,0)^2\lambda_s(0,0)^{-2}$. 
  \end{itemize}
  \end{theorem}

\begin{proof}
At first, we consider the case $0<r<\infty$. 

If $I_s=\varnothing$ and $I_u\neq \varnothing$, 
 $\lambda_s(p,q)$ is a constant, then we have $\partial_p d_s(0,0)=\dfrac{\partial_p h(0,0)}{\lambda_s(0,0)}$. 
Setting $R_r(p,q)=h(p,q)+r d_u(p, q)+r d_s(p, q)$, 
by using Proposition \ref{CMP-Proposition 4}, we obtain 
$\partial_{p}\left(R_{r}(0,0)\right)=r \partial_{p} d_{u}(0, 0)>0 $. 
Therefore $\nu_{0,0}$ is not the MM$r$NE. 

Similarly, the situation $I_u=\varnothing$ and $I_s\neq \varnothing$ will not occur.

 For the case $I_s\neq\varnothing$ and $I_u\neq \varnothing$, we have
\begin{align*}
\partial_p d_s(p,0)=-\partial_p \left(\frac{h(p,0)}{\lambda_s(p,0)}\right)=-\frac{\partial_p h(p,0)\cdot \lambda_s(p,0)-h(p,0)\cdot \partial_p \lambda_s(p,0)}{\lambda_s(p,0)^2}
\end{align*}
Through the inequality above and Proposition \ref{CMP-Proposition 4}, we can derive that
\begin{align*}
  \partial_{p}(R_r(0,0))=rh(0,0)\left(-\frac{\partial_p \lambda_u(0,0)}{\lambda_u(0,0)^2} + \frac{\partial_p \lambda_s(0,0)}{\lambda_s(0,0)^2}\right)=0 \Leftrightarrow  \frac{\partial_{p} \lambda_{u}(0, 0)}{\lambda_{u}(0, 0)^{2}} & =\frac{\partial_{p} \lambda_{s}(0, 0)}{\lambda_{s}(0, 0)^{2}}.
\end{align*}
Similarly, we could obtain
\begin{align*}
  \partial_{q}(R_r(0,0))=rh(0,0)\left(-\frac{\partial_q \lambda_u(0,0)}{\lambda_u(0,0)^2} + \frac{\partial_q \lambda_s(0,0)}{\lambda_s(0,0)^2}\right)=0 \Leftrightarrow  \frac{\partial_{q} \lambda_{u}(0, 0)}{\lambda_{u}(0, 0)^{2}} & =\frac{\partial_{q} \lambda_{s}(0, 0)}{\lambda_{s}(0, 0)^{2}}.
\end{align*}
Then $\partial_{p}(R_r(0,0))=0$ and $\partial_{q}(R_r(0,0))=0$ if and only if
\begin{align}\label{condition A}
  \frac{\partial_{p}\partial_{p} Q(0, 0)}{\lambda_{u}(0, 0)^{4}}=-\frac{\partial_{p} \partial_{q}Q(0, 0)}{\lambda_{u}(0, 0)^{2}\lambda_{s}(0, 0)^{2}}=\frac{\partial_{q} \partial_{q}Q(0, 0)}{\lambda_{s}(0, 0)^{4}}
\end{align}
For the sake of convenience, we set $a_0=\lambda_s(0,0)^2$, $b_0=\lambda_u(0,0)^2$. 
By applying \eqref{B-1}, \eqref{condition A} and Proposition \ref{CMP-Proposition 4}, we conclude that 
$\dfrac{a_0}{\lambda_s(0,0)^2}=\dfrac{b_0}{\lambda_u(0,0)^2}$ if and only if 
\begin{align*}                                                     
     \left.\frac{\rd^{2}}{\rd t^{2}} P\left(-t a_{0} \phi_{u}+t b_{0} \phi_{s}\right)\right|_{t=0} =&\left.\frac{\rd^{2}}{\rd t^{2}} Q\left(t a_{0} \phi_{u}, t b_{0} \phi_{s}\right)\right|_{t=0}     \\
  = & a_0^2\partial_{p}\partial_{p} Q(0, 0)+2a_0b_0\partial_{p}\partial_{q} Q(0, 0)+b_0^2\partial_{q}\partial_{q} Q(0, 0)= 0.
\end{align*} 
Therefore the function $-\phi_{u}+ \alpha\phi_{s}$ is cohomologous to a constant where $\alpha=b_0 a_0^{-1}>0$. 
Thus the proof of the case $0<r<\infty$ is complete. 

Finally, we turn to the case $r=\infty$. 
Since $\partial_ph(0,0)=\partial_qh(0,0)=0$, 
the proof of this case follows if we refine $R_\infty(p,q)=d_u(p,q)+d_s(p,q)$ and Repeat the argument above. 
Hence we complete the proof. 
\end{proof}

We are now ready to prove the following  rigidity result. 

\begin{theorem}\label{Anosov-infty}
  Let $f:\mathbb{T}^2\to \mathbb{T}^2$ be a $C^\infty$ area-preserving Anosov diffeomorphism,
 then the following statements are equivalent to each other.
 \begin{itemize}
\item[(1)]  the measure of maximal entropy is the measure maximizing $r$-neutralized entropy for some $r>0$; 
\item[(2)]  the measure of maximal entropy is the measure of maximal Hausdorff dimension; 
\item[(3)]  the ergodic measure of maximal Hausdorff dimension is the measure maximizing $r$-neutralized entropy for some $r>0$; 
\item[(4)]$f$ is $C^\infty$ conjugate to a total automorphism.
\end{itemize}
\end{theorem}

\begin{proof}
We will prove the theorem in three steps.

{\bf Step 1: (4) $\Rightarrow$ (1),(2), and (3). }
If (4) holds, the Lyapunov exponents of all invariant measures are the same, 
and the invariant measure is the MMHD if and only if it is the MME. 
By Lemma \ref{case 1} and Theorem \ref{rigidity results}, 
we establish (4) $\Rightarrow$ (1),(2), and (3).

{\bf Step 2: (1) $\Rightarrow$ (4) and (2) $\Rightarrow$ (4).}
Theorem \ref{varnothing} and the assumption that $f$ is area-preserving imply that 
$\phi_u$ and $\phi_s$ are respectively cohomologous to constants. 
Consequently, the Lyapunov exponents of all periodic orbits are the same. 
According to \cite[Theorem 1]{MM}, we derive (1) $\Rightarrow$ (4) and (2) $\Rightarrow$ (4).

To complete the proof, it remains to show (3) $\Rightarrow$ (4). 
On the contrary, without loss of generality, we suppose that $\phi_u$ is not cohomologous to a constant. 

{\bf Step 3: (3) $\Rightarrow$ (4). }
Let $\mu$ be the measure satisfying the condition (3). 
The area-preserving property of $f$ implies that
 $\mu$ is absolutely continuous with respect to the volume. 
Since all Anosov diffeomorphisms on two-dimensional manifolds are transitive \cite{Ne}, 
we conclude that $\mu$ is the equilibrium measure of the potential $-\phi_u(x)$ 
by \cite[Corollary 4.13]{Bo75}, i.e., $\mu=\nu(1,0)$. 
From condition in (3), we derive 
\begin{align*}
\partial_p h(1,0)=\partial_p \left(h(1,0)+rd_u(1,0)+rd_s(1,0) \right)-r\partial_p \left(d_u(1,0)+d_s(1,0) \right)=0-0=0. 
\end{align*}
However, Proposition \ref{CMP-Proposition 4} shows that 
$\partial_p h(1,0)=\partial_p\lambda_u(1,0)<0$.
This contradicts to the fact that $\partial_p h(1,0)=0$. 
 Hence, we obtain that $\phi_u$ is cohomologous to a constant. 
 Similarly, $\phi_s$ is cohomologous to a constant. 
 By \cite[Theorem 1]{MM}, we obtain (3) $\Rightarrow$ (4). 

Therefore, we finish the proof of this theorem. 
\end{proof}

\appendices

\section*{Appendix}
\section{Further results on $r$-neutralized entropy}\label{dimension theory}

  In this section, we focus on the variational principles related to 
  the $r$-neutralized local entropy. 
We develop a new variational principle for Borel probability measures. 
Afterwards, we propose a conjecture for invariant measures and prove the conjecture for
topological Markov shifts and linear Anosov automorphisms.

Let $(X,d)$ be a compact metric space and $f:X\to X$ a continuous map. 
For $ r>0 $, $ N \in \mathbb{N}$, $s \in \mathbb{R} $ and a nonempty subset $ Z \subset X $, we define
\[
  M_{N, r, d}^{s}(Z):=\inf \sum_{i \in I} e^{-n_{i} s}
\]
where the infimum is taken over all finite or countable covers $ \left\{B\left(x_{i}, n_i, e^{-n_{i} r}\right)\right\}_{i \in I} $ of
$ Z$  with $ n_{i} \geq N$ and $x_{i} \in X $ for each $i\in I$. 
Obviously, the limit $M_{r}^{s}(Z)=\lim\limits_{N \rightarrow \infty} M_{N, r,d}^{s}(Z)$ exists, 
as $M_{N, r,d}^{s}(Z)$ is increasing with respect to $N$.
The quantity $ M_{r,d}^{s}(Z) $ exhibits a critical behavior with respect to the parameter $s$, 
transitioning from $\infty$ to $0$ at a certain value. 
We define the {\bf $r$-neutralized Bowen topological entropy} as
\[
  h^{B}_{r,d}(T,Z):=\inf \left\{s: M_{r,d}^{s}(Z)=0\right\}=\sup \left\{s: M_{ r,d}^{s}(Z)=\infty\right\}.
\]

Motivated by the Frostman's lemma, we prove the following lemma. 

\begin{lemma}\label{Billingsley type}
  Let $(X,d)$ be a compact metric space and $f:X\to X$ a continuous map. 
  Let $\mu$ be a Borel probability measure on $X$ and $E$ a Borel subset of $X$ with $\mu(E)>0$. 
  Given $r>0$ and $0<s<\infty$, 
  if $\underline{h}_{\mu, d}^r(x)\geq s$ for all $x\in E$, then $h^{B}_{r,d}(T,E)\geq s$.
\end{lemma}

\begin{proof}
 Given $ \epsilon>0 $, for each $ N \geq 1 $, we define
  \[
    E_{N}:=\left\{x \in E: -\frac{1}{n}\log \mu\left(B(x,n, e^{-nr})\right)>s-\epsilon \quad \text { for all } n \geq N \right\} .
  \]
  Note that $\{E_N\}_{N\geq 1}$ forms an increasing sequence of sets.
  Clearly, there exists a positive integer $ N_{0}$ such that  $\mu\left(E_{N_{0}}\right)>0$.
  Then we have
  \begin{align}\label{Billingsley-(2)}
    \mu\left(B(x,n, e^{-nr})\right) \leq \mathrm{e}^{-(s-\epsilon) n}, \quad \text { for all } x \in E_N,  \text { and } n \geq N_0 .
  \end{align}
  Fix $N\geq N_0$. We consider an arbitrary cover 
  $\mathcal{F}=\left\{B(y_{i}, n_{i}, e^{-n_{i}r})\right\}_{i}$ of $E_{N_0}$ such that
  \[
    E_N \cap B\left(y_{i},n_{i}, e^{-n_{i}r}\right) \neq \varnothing, \quad n_{i} \geq N \geq N_0 \quad \text { for all } i \geq 1.
  \]
This choice ensures that each ball in the cover contains at least one point from $E_N$.

By applying \eqref{Billingsley-(2)}, we obtain
  \[
    \sum_{i \geq 1} \mathrm{e}^{-(s-\epsilon) n_{i}} \geq \sum_{i \geq 1} \mu\left(B(y_{i}, n_{i}, e^{-n_{i}r})\right) \geq \mu\left(E_{N_0}\right)>0.
  \]
  Therefore, we derive $ M_{N_0,r,d}^{s-\epsilon}\left(E_{N_0}\right) \geq \mu\left(E_{N_0}\right)>0  $ for all $ N\geq N_0 $, thus
  \[
    M_{r,d}^{s-\epsilon}\left(E_{N_0}\right)=\lim _{N \rightarrow \infty} M_{N,r,d}^{s-\epsilon}\left(E_{N_0}\right) \geq \mu\left(E_{N_0}\right)>0,
  \]
  which shows $h^{B}_{r,d}(T,E_{N_0}) \geq s-\epsilon$. 
Noting that $E_{N_0}$ is a subset of $E$,  
we conclude that
$  h^{B}_{r,d}(T,E) \geq h^{B}_{r,d}(T,E_{N_0}) \geq s-\epsilon$. 
By letting $\epsilon \rightarrow 0$, we complete the proof.
\end{proof}

It is worth mentioning that we can establish the following variational principle.

\begin{theorem}\label{variational principle}
  Let $(X,d)$ be a compact metric space, $f:X\to X$ a continuous self-map, 
  and $\mathcal{M}(X)$ the Borel probability measures on $X$. 
  Given $r>0$, then
  \[
    h_{\text {r,d}}^{B}(f)=\sup \left\{\underline{h}_{\mu,d}^r(f): \mu \in \mathcal{M}(X)\right\}.
  \]
\end{theorem}

As for the proof, the argument of \cite{[FH12]} can be directly applied here, if the Bowen balls are replaced by $r$-neutralized Bowen balls.
Therefore, we omit the detailed proof of the above theorem. 
A similar result for neutralized local entropy can be found in \cite{[YCZ23]}. 

The above variational principle is established for the Borel probability measures. 
We conjecture that a similar variational principle should exist for invariant measures. 
Before stating our conjecture, 
we introduce the following notion.

Let $(X,d)$ be a compact metric space and $f:X\to X$ a continuous self-map. 
Given $r>0$ and $n\in \mathbb{Z}$, 
a set $ E \subset X $ is said to be $ (n, e^{-nr}) $-spanning if $ X \subset \bigcup_{x \in E} B(x, n, e^{-nr}) $.
Let $ S_{d}(f, n, e^{-nr}) $ be the minimal cardinality of an $ (n,e^{-nr})  $-spanning set.
For $r>0$, the {\bf $r$-neutralized topological entropy} $h_{d}^r(f)$ and
  {\bf lower $r$-neutralized topological entropy} $\underline{h}^r_{d}(f)$ are defined by
\begin{align*}
  h_{d}^r(f):=\varlimsup_{n \rightarrow \infty} \frac{1}{n} \log S_{d}(f, n, e^{-nr}),\quad  \underline{h}^r_{d}(f):=\varliminf\limits_{n \rightarrow \infty}\frac{1}{n}\log S_{d}(f, n, e^{-nr}).
\end{align*}

Motivated by the classical variational principle, we propose the following conjecture.
\begin{conjecture}\label{con-conj1}
  If $f:X\to X$ is a homeomorphism of a compact metric space $(X,d)$, then for any $r>0$, we have
  \begin{align*}
    h_{d}^r(f)=\sup\{h_{\nu, d}^r(f):\mu\in \mathcal{M}(f,X)\}, \quad \underline{h}_{d}^r(f)=\sup\{\underline{h}_{\nu, d}^r(f):\mu\in \mathcal{M}(f,X)\}.
  \end{align*}
\end{conjecture}
We will present two examples of topological Markov chains and linear Anosov automorphisms below 
to illustrate why we believe the aforementioned conjecture holds.

At first, we recall some notions of topological Markov chains.

Let $\mathcal{G}$ be a directed graph with a finite collection of vertices $\mathcal{V}$ such that
every vertex has at least one edge coming in and at least one edge coming out. The topological Markov shift associated to $\mathcal{G}$ is the set
$$
  \Sigma=\Sigma(\mathcal{G}):=\{(v_i)_{i\in\mathbb{Z}}\in\mathcal{V}^{\mathbb Z}:v_i\to v_{i+1}\textrm{ for all }i\}.
$$
Fix $ \theta \in(0,1) $, we equip $\Sigma$ with the metric $ d_\theta(x, y)=\theta^{N(x, y)} $,
where $ N(x, y) $ is the distinguished time of $ x $ and $ y $,
$ N(x, y)=\min \left\{|n|: x_{n} \neq y_{n}, n \in \mathbb{Z}\right\}$.
The left shift map $\sigma:\Sigma\to \Sigma$, 
defined by $\sigma[(v_i)_{i\in\mathbb{Z}}]=(v_{i+1})_{i\in\mathbb{Z}}$, 
is a homeomorphism on $\Sigma$.
Since $\mathcal{G}$ is finite, $\Sigma$ is compact.

\begin{proposition}\label{Symbolic-proposition}
  For the topological Markov shift $(\Sigma(\mathcal{G}),d_\theta,\sigma)$ and a $\sigma$-invariant measure $\mu$, we have
  \begin{align*}
    h_{\mu,d_\theta}^r(\sigma)= \underline{h}_{\mu,d_\theta}^r(\sigma)=(1-\frac{2r}{\log \theta})h_\mu(\sigma), \\
    h_{d_\theta}^r(\sigma)= \underline{h}_{d_\theta}^r(\sigma)=(1-\frac{2r}{\log \theta})h_{\text{top}}(\sigma).
  \end{align*}
\end{proposition}

\begin{proof}
  Fix $\epsilon>0$. By setting $N:=\left[\dfrac{\log \epsilon}{\log \theta}\right]$ and 
  $M_n:=\left[-\dfrac{rn}{\log \theta}\right]$, we obtain
  \begin{align*}
    B(x,n,\epsilon) & =\{y\in X: x_i=y_i,\, -N\leq i <N+n \};      \\
    B(x,n,e^{-nr})  & =\{y\in X: x_i=y_i,\, -M_n \leq i <M_n+n \}.
  \end{align*}
  Thus we obtain $f^{M_n-N}B(x,n+2M_n-2N,\epsilon)=B(x,n, e^{-rn})$. Then for $\mu$-a.e. $x\in X$,
  \begin{align*}
    h^r_{\mu,d_\theta}(\sigma,x) & =\varlimsup_{n\to\infty}-\frac{\log \mu(B(x,n,e^{-nr}))}{n}=\varlimsup_{n\to\infty}-\frac{\log \mu(B(x,n+2M_n-2N,\epsilon))}{n} \\
                                 & =\varlimsup_{n\to\infty}-\frac{\log \mu(B(x,n, \epsilon))}{n}\left(1-\frac{2r}{\log \theta}\right).
  \end{align*}
  By letting $\epsilon\to 0$ and integrating over $\mu$, 
  we obtain $ h^r_{\mu,d_\theta}(\sigma) =\left(1-\dfrac{2r}{\log \theta}\right)h_\mu(\sigma)$. 
  Similarly, we could derive 
  \begin{align*}
  \underline{h}^r_{\mu,d_\theta}(\sigma)=\left(1-\dfrac{2r}{\log \theta}\right)h_\mu(\sigma)\quad \text{ and }\quad h_{d_\theta}^r(\sigma)= \underline{h}_{d_\theta}^r(\sigma)=\left(1-\dfrac{2r}{\log \theta}\right)h_{\text{top}}(\sigma). 
  \end{align*}
  Thus the proof is complete.
\end{proof}

By applying the classical variational principle, we conclude that the conjecture holds for 
topological Markov shifts. 
Additionally, we derive that the
$r$-neutralized entropy depends on the metrics of the systems in general cases, 
as previously mentioned. This dependence on the metric is a key difference 
from classical entropy theory. 

Before giving the next example, we introduce the following definitions. 

Let $(X,d)$ be a compact metric space and $f:X\to X$ a homeomorphism. 
Motivated by \cite{[K-IHES]},
given $ r>0$, $n\in \mathbb{N}$ and $0<\delta\leq 1 $, 
we use $ S_d(n, e^{-nr}, \delta) $ to denote
the minimal number of $r$-neutralized Bowen balls which cover the set of measure more than or equal to $1-\delta $.
The {\bf $r$-neutralized Katok entropy} $h^{K,r}_{\mu,d}(f) $ and
the {\bf lower $r$-neutralized Katok entropy} $\underline{h}^{K,r}_{\mu,d}(f)$ with respect to $f$-invariant measure $\mu$ are defined by
\[
  h^{K,r}_{\mu,d}(f)= \lim_{\delta \to 0}\varlimsup_{n \rightarrow \infty} \frac{\log S_d(n, e^{-nr}, \delta)}{n}, \quad
  \underline{h}^{K,r}_{\mu,d}(f)= \lim_{\delta \to 0}\varliminf_{n \rightarrow \infty} \frac{\log S_d(n, e^{-nr}, \delta)}{n}.
\]
Obviously, we have the lemma as follows. 
\begin{lemma}\label{K-lemma}
  Let $(X,d)$ be a compact metric space and $f:X\to X$ a homeomorphism.
  For any ergodic $f$-invariant measure $\mu$, we have
  \[
    h_{\mu,d}^{r}(f)\geq h_{\mu,d}^{K,r}(f)\geq \underline{h}_{\mu,d}^{K,r}(f) \geq\underline{h}_{\mu,d}^{r}(f).
  \]
\end{lemma}

We are now ready to prove the following proposition. 

\begin{proposition}
The Conjecture \ref{con-conj1} holds for linear Anosov diffeomorphisms on torus. 
\end{proposition}

\begin{proof}
Given $n\in\mathbb{N}$, let $T:(\mathbb{T}^n,d)\to(\mathbb{T}^n,d)$ be a linear Anosov diffeomorphisms
  and $m$ is the normalized Haar measure on $\mathbb{T}^n$. 
  The Lyapunov exponents for any $T$-invariant measures are the same, 
  we use $\lambda_1,\cdots, \lambda_n$ to denote them. 
Let $\mu$ denote the MME of $(T, \mathbb{T}^n)$. 
By $h_\mu(T)=h_{\text{top}}(T)=\sum_{i=1}^n\max\{0,\log|\lambda_i|\}$, 
we have $\operatorname{dim}_H m=n$. Then 
\begin{align*}
h_{\mu,d}^r(T)=h_{\mu}(T)+r\operatorname{dim}_H \mu=\sum_{i=1}^n\max\{0,\log|\lambda_i|\}+rn. 
\end{align*}
We use $\operatorname{Vol}(B)$ to represent the volume of the set $B\subset \mathbb{T}^n$. 
Fix $\epsilon>0$, there are constants $C>0$ and $N_0\in \mathbb{N}$ such that for any $x\in \mathbb{T}^n$ and $k\geq N_0$, 
\begin{align*}
\operatorname{Vol}\left(B(x,k,\frac{1}{2}e^{-kr})\right)=C^{\pm 1}\exp(-k\sum_{i=1}^n\max\{0,\log|\lambda_i|\}-knr\pm kn\epsilon). 
\end{align*}
Obviously, 
 there exists a $(k,e^{-kr})$-spanning set $E\subset \mathbb{T}^2$ 
such that for any $x,y\in E$ with $x\neq y$, $y\notin B\left(x,k,\frac{1}{3}e^{-kr}\right)$. 
Therefore we conclude that 
$$S_d(T,k,e^{-kr})\leq C\exp\left(k\sum_{i=1}^n\max\{0,\log|\lambda_i|\}+knr+kn\epsilon\right).$$ 
By letting $\epsilon\to 0$, we derive
$h_d^r(T)\leq \sum\limits_{i=1}^n\max\{0,\log|\lambda_i|\}+nr= h_{\mu,d}^r(T)$. 

On the other hand, 
by Theorem \ref{BPS-theorem-1} and Lemma \ref{K-lemma},
for any ergodic $T$-invariant measure $\nu$, we have
\begin{align*}
  h_d^r(T)\geq \underline{h}_d^r(T)\geq \underline{h}^{K,r}_{\nu,d}(T)=h_{\nu,d}^{r}(T)=\underline{h}_{\nu,d}^{r}(f).
\end{align*}
According to \eqref{nonergodic-ergodic} and the inequality above, we derive
\begin{align*}
  h_{d}^r(T)\geq \underline{h}_{d}^r(T)\geq \sup\{h_{\nu, d}^r(T):\mu\in \mathcal{M}_e(f,X)\}=\sup\{h_{\nu, d}^r(T):\mu\in \mathcal{M}(f,X)\},
\end{align*}
where $\mathcal{M}_e(T,X)$ is the set of ergodic measures in $\mathcal{M}(T,X)$. 

Based on the previous discussion, we complete the proof of the proposition. 
\end{proof}

These results provide evidence in support of our conjecture and suggest that
$r$-neutralized entropy exhibits complex properties.


\begin{thebibliography}{99}

 \bibitem{AKM}  R. L. Adler, A. G. Konheim, M. H. McAndrew, {\it Topological entropy}. Trans. Amer. Math. Soc. {\bf 114} (1965), 309--319.

\bibitem{Anosov} D. V. Anosov, {\it Tangential fields of transversal foliations in U-systems}. Math. Notes {\bf 2} (1967), 818--823. 

  \bibitem{BW03} L. Barreira and C. Wolf, {\it Measures of maximal dimension for hyperbolic diffeomorphisms}. Comm. Math. Phys. {\bf 239} (2003) 93--113.

\bibitem{BW06} L. Barreira and C. Wolf, {\it Pointwise dimension and ergodic decompositions}. Ergodic Theory Dynam. Systems {\bf 26} (2006) 653--671. 

\bibitem{BG11} L. Barreira and K. Gelfert, {\it Dimension estimates in smooth dynamics: a survey of recent results}. Ergodic Theory Dynam. Systems {\bf 31}(2011), 641--671.

  \bibitem{BPS-99-annals} L. Barreira, Ya. Pesin and J. Schmeling, {\it Dimension and product structure of hyperbolic measures}. 
  Ann. of Math. (2) {\bf 149} (1999), 755--783. 

\bibitem{BP-nonuniform} L. Barreira, Ya. Pesin, {\it Nonuniform Hyperbolicity}. Encyclopedia Math. Appl. {\bf 115}, Cambridge University Press, Cambridge (2007). 

\bibitem{ORH23} S. Ben Ovadia and F. Rodriguez Hertz, {\it Neutralized local entropy and dimension bounds for invariant measures}. Int. Math. Res. Not. {\bf 11} (2024), 9469--9481.

\bibitem{ORH23-2} S. Ben Ovadia and F. Rodriguez Hertz, {\it Exponential volume limits}. arXiv: 2308. 03910.

\bibitem{O24}  S. Ben Ovadia, {\it Tubular dimension: Leaf-wise asymptotic local product structure, and entropy and volume growth}. arXiv: 2402.02496.

\bibitem{Bo73} R. Bowen, {\it Topological entropy for noncompact sets}, Trans. Amer. Math. Soc. {\bf 184} (1973) 125--136. 

\bibitem{Bo75} R. Bowen, {\it Equilibrium states and the ergodic theory of Anosov diffeomorphisms}. Lecture Notes in Math. {\bf 470}, Springer, Berlin, 1975. 

  \bibitem{[BK83]} M. Brin and A. Katok, {\it On local entropy}. Lecture Notes in Math. {\bf 1007} (1983), 30--38.
  
\bibitem{BCFT18} K. Burns, V. Climenhaga, T. Fisher, and D. J. Thompson, {\it Unique equilibrium states for geodesic flows in nonpositive curvature}. 
Geom. Funct. Anal. {\bf 28} (2018), 1209--1259. 

\bibitem{CT-advance} V. Climenhaga and D. J. Thompson, {\it Unique equilibrium states for flows and homeomorphisms with non-uniform structure}. Adv. Math. {\bf 303} (2016), 745--799.

\bibitem{CPZ-JMD} V. Climenhaga, Y. Pesin and A. Zelerowicz, {\it Equilibrium measures for some partially hyperbolic systems}. J. Mod. Dyn. {\bf 16} (2020) 155--205.

\bibitem{CKW20} V. Climenhaga, G. Knieper, and K. War, {\it Uniqueness of the measure of maximal entropy for geodesic flows on certain manifolds without conjugate points}. 
Adv. Math. {\bf 376} (2021), 107452, 44. 
  
\bibitem{FCZ20} J. Fang, and Y. Cao and Y, Zhao, {\it Measure theoretic pressure and dimension formula for non-ergodic measures}. Discrete Contin. Dyn. Syst. {\bf 40} (2020), 2767--2789. 

  \bibitem{[FH12]} D. Feng and W. Huang, {\it Variational principles for topological entropies of subsets}. J. Funct. Anal. {\bf 263} (2012), 2228--2254. 
  
\bibitem{Franks} J. Franks, {\it Anosov diffeomorphisms on tori}. Trans. Amer. Math. Soc. {\bf 145} (1969), 117--124.

\bibitem{HHW17} H. Hu, Y. Hua, and W. Wu, {\it Unstable entropies and variational principle for partially hyperbolic diffeomorphisms}. Adv. Math. {\bf 321} (2017), 31--68.

\bibitem{Kad} S. Kadyrov, {\it Effective equidistribution of periodic orbits for subshifts of finite type}. Colloq. Math. {\bf 149} (2017), 93--101. 

\bibitem{Kn98} G. Knieper, {\it The Uniqueness of the Measure of Maximal Entropy for Geodesic Flows on Rank $1$ Manifolds}. Ann. of Math. (2) {\bf 148} (1998), 291--314. 

\bibitem{[K-IHES]} A. Katok, {\it Lyapunov exponents, entropy and periodic orbits for diffeomorphisms}. Inst. Hautes \'{E}tudes Sci. Publ. Math, {\bf 51} (1980), 137--174. 
  

\bibitem{MM} J. M. Marco and R. Moriy\'{o}n, {\it Invariants for smooth conjugacy of hyperbolic dynamical systems. {III}}. Comm. Math. Phys. {\bf 112} (1987), 317--333.


 
  \bibitem{Kolmogorov}  A. N. Kolmogorov, {\it A new invariant for transitive dynamical systems}. Dokl. Akad. Nauk SSSR. {\bf 119} (1958), 861--864.
  
\bibitem{LM85} F. Ledrappier and M. Misiurewicz, {\it Dimension of invariant measures for maps with exponent zero}. 
Ergodic Theory Dynam. Systems {\bf 5} (1985), 595--610.

  \bibitem{LY-85-annals-1} F. Ledrappier and L.-S. Young, {\it The metric entropy of diffeomorphisms. {I}. 
  {C}haracterization of measures satisfying {P}esin's entropy formula}. Ann. of Math. (2)  {\bf 122} (1985), 509--539. 
  
  \bibitem{LY-85-annals-2} F. Ledrappier and L.-S. Young, {\it The metric entropy of diffeomorphisms. {II}. 
  {R}elations between entropy, exponents and dimension}. Ann. of Math. (2) {\bf 122} (1985), 540--574. 
  
\bibitem{Ne} S. Newhouse, {\it On codimension one Anosov diffeomorphisms}. Am. J. Math. {\bf 92} (1970), 761--770. 
 
  \bibitem{Oseledec} V. I. Oseledec, {\it A multiplicative ergodic theorem, Lyapunov characteristic numbers for dynamical systems}. 
  Trudy Moskov. Mat. Ob\v{s}\v{c}. {\bf 19} (1968), 179-210, Engl. translation.   

  
\bibitem{Rams} M. Rams. Measures of maximal dimension for linear horseshoes. Real Anal. Exchange {\bf 31} (2005), 55--62.

\bibitem{RHRHTU12} F. Rodriguez Hertz, M. A. Rodriguez Hertz, A. Tahzibi, and R. Ures, 
{\it Maximizing measures for partially hyperbolic systems with compact center leaves}, Ergodic Theory Dynam. Systems {\bf 32} (2012), 825--839. 

 \bibitem{Ruelle78} D. Ruelle, {\it Thermodynamic formalism}. 
 Encyclopedia of Mathematics and its Applications {\bf 5},
Addison-Wesley Publishing Co., Reading, MA (1978). 

\bibitem{Smale} S. Smale, {\it Differentiable dynamical systems}. Bull. Amer. Math. Soc. {\bf 73} (1967), 747--817. 

\bibitem{UW08} M. Urbanski and C. Wolf, {\it Ergodic theory of parabolic horseshoes}. Comm. Math. Phys. {\bf 281} (2008), 711--751.

\bibitem{[PW]} P. Walters, {\it An introduction to ergodic theory}. Grad. Texts in Math. {\bf 79}, Springer, New York (1982). 
  
\bibitem{Wan} T. Wang, {\it Unique equilibrium states, large deviations and Lyapunov spectra for the Katok map}. Ergodic Theory Dynam. Systems {\bf 7} (2021), 2182--2219. 

\bibitem{Wen} L. Wen, {\it Differentiable dynamical systems}. Grad. Stud. Math. {\bf 173}, American Mathematical Society, Providence, RI (2016). 

  \bibitem{[YCZ23]} R. Yang, E. Chen and X. Zhou, {\it Variational principle for neutralized Bowen topological entropy}. Qual. Theory Dyn. Syst. {\bf 23}, 162 (2024). 
  
\bibitem{Y82}  L.-S. Young, {\it Dimension, entropy and Lyapunov exponents}. Ergodic Theory Dynam. Systems {\bf 2} (1982) 109--124.

\end{thebibliography}
\end{document}